\documentclass[letterpaper,leqno,11pt]{amsart}
\usepackage{latexsym,amssymb,amsmath,amsthm,color}
\usepackage{amsfonts}

\advance\textwidth by 2.5cm 
\advance\oddsidemargin by -1.6cm \advance\evensidemargin by -1.6cm

\usepackage{mathrsfs}
\usepackage[normalem]{ulem}
\usepackage[toc,page]{appendix}
\usepackage[colorlinks,citecolor=blue,urlcolor=blue]{hyperref}

\newcommand{\rH}{\mathrm{ H}}
\newcommand{\rA}{\mathrm{ A}}

\newcommand{\divv}{{\rm div}\,}

\definecolor{light}{gray}{.96}

\newcommand\think[1]{}
\newcommand\delb[1]{}

\numberwithin{equation}{section}

\newcommand\Eb{\mathbb{E}}

\newcommand\N{\mathbb{N}}

\def\v{\mathrm{v}}

\newcommand\R{\mathbb{R}}

\newcommand\dom{\mathcal{O}}

\newcommand\E{\mathbb{E}}

\newcommand\T{\mathbb{T}^2}

\newcommand\tp{\tilde{\mathbb{P}}}
\newcommand\hp{\hat{\mathbb{P}}}

\newcommand\tom{\tilde{\Omega}}
\newcommand\ho{\hat{\Omega}}
\newcommand\rV{\mathrm{V}}
\newcommand\rD{\mathrm{D}}

\newcommand{\tunk}{{\tilde{u}}_{{n}_{k}}}
\newcommand{\tu}{\tilde{u}}

\newcommand{\ccal}{\mathcal{C}}

\newcommand{\kcal}{\mathcal{K}}

\newcommand{\zcal}{\mathcal{Z}}

\newcommand{\nat}{\mathbb{N}}

\newcommand\un{u_n(t)}
\newcommand\uns{u_n(s)}
\newcommand\tus{\tilde{u}(s)}

\newcommand\tuns{\tilde{u}_n(s)}
\newcommand\tusi{\tilde{u}(\sigma)}
\newcommand\tunsi{\tilde{u}_n(\sigma)}
\newcommand\tun{\tilde{u}_n}

\newtheorem{theorem}{Theorem}[section]

\newtheorem{alg}[theorem]{Algorithm}

\newtheorem{corollary}[theorem]{Corollary}

\newtheorem{defn}[theorem]{Definition}

\newtheorem{example1}[theorem]{Example}
\newtheorem{exercise1}[theorem]{Exercise}
\newtheorem{lemma}[theorem]{Lemma}

\newtheorem{remark1}[theorem]{Remark}

\newtheorem{digression1}[theorem]{Digression}

\newenvironment{definition}{
\begin{defn}
	\normalfont}{
\end{defn}
}

\newenvironment{remark}{
\begin{remark1}
	\normalfont}{
\end{remark1}
}

\newtheoremstyle{AppALem}{1}{1}
  {\itshape}{0pt}{\bfseries}{.}{ }
   {\thmname{Lemma }\thmnumber{A.{#2}}{\thmnote{}}}
   \theoremstyle{AppALem}

\newtheoremstyle{AppBLem}{1}{1}
  {\itshape}{0pt}{\bfseries}{.}{ }
   {\thmname{Lemma }\thmnumber{B.{#2}}{\thmnote{}}}
   \theoremstyle{AppBLem}

\newtheoremstyle{AppBRem}{1}{1}
  {\itshape}{0pt}{\bfseries}{.}{ }
   {\thmname{Remark }\thmnumber{B.{#2}}{\thmnote{}}}
   \theoremstyle{AppBRem}

\newtheoremstyle{AppBCor}{1}{1}
  {\itshape}{0pt}{\bfseries}{.}{ }
   {\thmname{Corollary }\thmnumber{B.{#2}}{\thmnote{}}}
   \theoremstyle{AppBCor}

\definecolor{darkred}{rgb}{0.9,0.1,0.1}

\definecolor{darkblue}{rgb}{0.1,0.1,0.9}

\begin{document}
\title{Stochastic Constrained Navier-Stokes Equations on $\T$}
\author{Zdzis\l aw Brze\'{z}niak}
\address{Department of Mathematics, University of York, Heslington, York,
YO10 5DD, UK}
\email{zdzislaw.brzezniak@york.ac.uk}
\thanks{The research of Zdzislaw Brze{\'z}niak has been partially supported by the Leverhulme project grant ref no RPG-2012-514. The research of Gaurav Dhariwal was supported by Department of Mathematics, University of York.}

\author{Gaurav Dhariwal}
\address{Institue of Analysis and Scientific Computing, Vienna University of Technology, Wiedner Haupstrasse 8--10, 1040 Wien, Austria}
\email{gaurav.dhariwal@tuwien.ac.at}


\date{\today}


\begin{abstract}
We study constrained 2-dimensional Navier-Stokes Equations driven by a multiplicative Gaussian noise in the Stratonovich form.
In the deterministic case \cite{[BDM16]}  we  showed the existence of   global solutions only  on a two dimensional torus and hence we concentrated on such a case here.
We prove the existence of a martingale solution and later using Schmalfuss idea \cite{[Schmalfuss97]} we show the pathwise uniqueness of the solutions.
We also establish the existence of a strong solution using a Yamada-Watanabe type result from Ondrej\'{a}t \cite{[Ondrejat04]}.
\end{abstract}
\maketitle


\section{Introduction}
\label{s:1}

In the present article we consider the stochastic Navier-Stokes equations
\begin{equation}
\label{eq:1.1}
du + \left[(u \cdot \nabla )u - \nu \Delta u + \nabla p\right]\,dt = \sum_{j=1}^m (c_j \cdot \nabla) u \circ dW_j(t), \quad t \in [0,\infty)
\end{equation}
in $\dom=[0,2\pi]^2$ with periodic boundary conditions and
with the incompressibility condition
\[ \mathrm{div}\, u = 0.\]
This problem can be identified as a problem on a two-dimensional torus $\mathbb{T}^2$ what we will assume to be our case.
Here  $u : [0,\infty) \times \dom \to \R^2$ and $p : [0,\infty) \times \dom \to \R$ represent the velocity and the pressure of the fluid. Furthermore $\sum_{j=1}^m(c_j\cdot \nabla)u \circ dW_j(t)$
stands for the random forcing, where $c_j$, $j=1,\cdots, m$,  are divergence free $\R^2$-valued vectors (so that the corresponding transport operators $\tilde{C}_ju := (c_j \cdot \nabla)u$ are skew symmetric in $L^2(\T, \R^2)$) and  $W_j$, $j = 1, \dots, m$ are independent $\R-$valued standard Brownian Motions.

The above problem projected on $\mathrm{H} \cap \mathcal{M}$ can be written in an abstract form as the following initial value problem
\begin{equation}
\label{eq:1.2}
\begin{cases}
du(t) + \nu \mathrm{A} u(t)\, dt + B(u(t))\, dt = \nu |\nabla u(t)|_{L^2}^2 u(t) \,dt + \sum_{j=1}^m C_j u(t) \circ dW_j(t), \; t \in [0,T], \\
u(0) = u_0,
\end{cases}
\end{equation}
where $\mathrm{H}$ is the space of square integrable, divergence free and mean zero vector fields on $\dom$ and
\[ \mathcal{M} = \{ u \in \mathrm{H} : |u|_{L^2} = 1\}.\]
Here $\mathrm{A}$ and $B$ are appropriate maps corresponding to the Laplacian and the nonlinear term respectively, appearing in the Navier-Stokes equations, see Section \ref{s:2} and $C_j = \Pi (\tilde{C}_j)$, where $\Pi : L^2(\T, \R^2) \rightarrow \rH$ is the Leray-Helmholtz projection operator \cite{[Temam79]} that projects the square integrable vector fields onto the divergence free vector field.\\

We prove the existence and uniqueness of a strong solution. The construction of a solution is based on the classical Faedo-Galerkin approximation, i.e.
\begin{equation}
\label{eq:1.3}
\begin{cases}
du_n(t) = -\left[P_n \mathrm{A} u_n(t) + P_nB(u_n(t)) - |\nabla u_n(t)|^2_{L^2} u_n(t) \right]\,dt \\
\quad \quad \quad \quad + \sum_{j=1}^m P_n C_j u_n(t) \circ dW_j(t), \quad t \in [0,T],\\
u_n(0) = \dfrac{P_n u_0}{|P_n u_0|}
\end{cases}
\end{equation}
given in Section~\ref{s:5}. Let us point out that without the normalisation of the initial condition in the above problem \eqref{eq:1.3}, the solution may not be a global one, even in the deterministic case. The crucial point is to prove suitable uniform a'priori estimates on the sequence $u_n$. We will prove that the following estimates hold
\[\sup_{n \ge 1} \E \left[\int_0^T |u_n(s)|^2_{\mathrm{D}(\mathrm{A})}\,ds \right] < \infty,\]
and
\[\sup_{n \ge 1} \E \left(\sup_{0 \le s \le T} \|u_n(s)\|^{2p}_{\mathrm{V}} \right) < \infty,\]
for $p \in [1, 1 + \frac{1}{K_c^2})$, where  $\mathrm{D}(\mathrm{A})$ is the domain of the Stokes operator and $\mathrm{V}=D(\mathrm{A}^{1/2})$,  see Section~\ref{s:2} for precise definitions and
the positive constant $K_c$ is defined in \eqref{eqn-K_c}.\delb{ Assumption $(A.1)$.}

In Theorem~\ref{thm5.1} we prove the existence of a martingale solution using the tightness criterion in the topological space  $\mathcal{Z}_T = \ccal([0,T]; \rH) \cap L^2_{\mathrm{w}}(0,T; \rD(\mathrm{A}))\cap L^2(0,T; \rV) \cap \ccal([0,T]; \rV_{\mathrm{w}})$ showing that the trajectories of the solution lie in $\ccal([0,T]; \rV_{\mathrm{w}})$ but later on, in Lemma~\ref{lemma6.1} we show that in fact the trajectories  lie in $\ccal([0,T]; \rV)$.\\

Our work is an extension of a recent article by the two authours and Mauro Mariani \cite{[BDM16]} from the deterministic to a stochastic setting. More information and motivation can also be found therein. Let us recall that already in the deterministic setting, we have been able to prove the global existence of solutions for Constrained Navier-Stokes equations (CNSEs) only with periodic boundary conditions and this is why we have concentrated here on such a case. A similar problem for  stochastic heat equation with polynomial drift but with a different type of noise has recently been a subject of a PhD thesis by Javed Hussain \cite{[Hussain15]}. It's remarkable that in that case the result holds for Dirichlet boundary conditions as well.

We consider the noise of gradient type in the Stratonovich form \eqref{eq:1.1}. The structure of noise is such that it is tangent to the manifold $\mathcal{M}$ just like the non-linear part from Navier-Stokes and hence there is no contribution to the equation \eqref{eq:1.2} because of the constraint. In the deterministic setting \cite{[BDM16]} we proved the existence of a global solution by proving the existence of a local solution using Banach Fixed Point Theorem; and no explosion principle, i.e enstrophy ($\rV-$ norm) of the solution remains bounded. We can't take the similar approach in the stochastic setting as one can't prove the existence of a local solution using the Banach Fixed Point Theorem and hence we switch to more classical approach of proving the existence of a solution using the Faedo-Galerkin approximation.

We consider the Faedo-Galerkin approximation \eqref{eq:1.3} of \eqref{eq:1.2}. We prove that each approximating equation has a global solution. One can show that for every $n \in \N$ global solution to \eqref{eq:1.3} exist for all domains, in particular for Dirichlet boundary conditions. But in order to obtain a'priori estimates as in Lemma~\ref{lemma5.5}, we need to consider the Navier-Stokes Equations (NSEs) on a two dimensional torus $\mathbb{T}^2$ (i.e. the NSEs with the periodic boundary conditions).

In order to prove that the laws of the solution of these approximating equations are tight on $\mathcal{Z}_T$ (defined in \eqref{eq:3.1}), apart from a'priori estimates we also need the Aldous condition, Definition~\ref{defn3.4}. After proving that the laws are tight in Lemma~\ref{lemma5.6}, by the application of the Jakubowski-Skorokhod Theorem and the martingale representation theorem we prove Theorem~\ref{thm5.1}. The paper is organised in the following way$\colon$

In Section~\ref{s:2} we introduce some functional spaces and certain linear operators along with the well-established estimates. Stochastic Constrained Navier-Stokes Equations (SCNSEs) are introduced in Section~\ref{s:4} along with the definitions of a martingale solution and strong solution and all the important results of this paper. Section~\ref{s:3} contains all the well-known and already established results regarding compactness. In Section~\ref{s:5} we establish certain estimates on the way to prove Theorem~\ref{thm5.1}. We conclude the paper by proving the existence and uniqueness of a strong solution using the results from Ondrej\'{a}t \cite{[Ondrejat04]} in Section~\ref{s:6}.

\section{Functional setting}
\label{s:2}

Let $\dom \subset \R^2$ be a bounded domain with periodic boundary conditions. Let $p \in [1, \infty)$ and let $\mathbf{L}^p(\dom)=L^p(\dom, \R^2)$ denote the Banach space of Lebesgue measurable $\R^2$-valued $p$-th power integrable functions on the set $\dom$. The norm in $\mathbf{L}^p(\dom)$ is given by
\[|u|_{L^p} := \left( \int_\dom |u(x)|^p\,dx\right)^{\frac1p}, \quad u \in \mathbf{L}^p(\dom).\]

By $\mathbf{L}^\infty(\dom)=L^\infty(\dom, \R^2)$ we denote the Banach space of lebesgue measurable essentially bounded $\R^2$-valued functions defined on $\dom$. The norm is given by
\[ |u|_{\mathbf{L}^\infty(\dom)} := \mathrm{esssup}\left\{|u(x)|, x \in \dom \right\}, \quad u \in \mathbf{L}^\infty(\dom).\]

If $p =2$, then $\mathbf{L}^2(\dom)=L^2(\dom, \R^2)$ is a Hilbert space with the scalar product given by
\[\langle u, \v \rangle_{L^2} := \int_\dom u(x) \cdot \v(x) \,dx, \quad u,\v \in \mathbf{L}^2(\dom).\]

Let $k \in \mathbb{N}$, $p \in [1, \infty)$. By $\mathbf{W}^{k,p}(\dom)=W^{k,p}(\dom, \R^2)$ we denote the Sobolev space of all $u \in \mathbf{L}^p(\dom)$ for which there exist weak derivatives $D^\alpha u \in \mathbf{L}^p(\dom)$, $|\alpha| \leq k$. For $p = 2$, we will write $W^{k,2}(\dom, \R^2) =: H^k$ and will denote it's norm by $\|\cdot\|_{H^k}$. In particular $H^1$ is a Hilbert space with the scalar product given by
\[\langle u, \v \rangle_{H^1} := \langle u, \v \rangle_{L^2} + \langle \nabla u, \nabla \v \rangle_{L^2}, \quad u, \v \in H^1(\dom).\]

Let $\ccal^\infty_c(\dom, \R^2)$ denote the space of all $\R^2-$valued functions of class $\ccal^\infty$ with compact support contained in $\dom$. We introduce the following spaces:
\begin{equation}
\label{eq:2.1}
\begin{split}
\mathcal{V} & = \left\{u \in \ccal^\infty_c(\dom, \R^2) : \mathrm{div}\,u = 0 \right\},\\
\mathbb{L}^2_0 & = \left\{u \in L^2(\T, \R^2) : \int_{\T} u(x)\,dx = 0 \right\},\\
{\mathrm{H}} & = \left\{ u \in \mathbb{L}^2_0 :\mathrm{div}\,u = 0 \right\},\\
\mathrm{V} &= H^1 \cap {\mathrm{H}}.
\end{split}
\end{equation}
We endow ${\mathrm{H}}$ with the scalar product and norm of $L^2$ and denote it by
\[\langle u , \v \rangle_{{\mathrm{H}}} := \langle u, \v \rangle_{L^2}, \quad \quad  |u|_{{\mathrm{H}}} := |u|_{L^2},\quad u,\v \in {\mathrm{H}}.\]
We equip the space $\mathrm{V}$ with the scalar product $\langle u, \v \rangle_\rV := \langle \nabla u, \nabla \v \rangle_{{\mathrm{H}}}$ and norm $\|u\|_{\mathrm{V}}, u, \v \in \mathrm{V}$.

One can show that in the case of $\dom = \T$, $\mathrm{V}$-norm $\|\cdot \|_{\mathrm{V}}$, and $H^1$-norm $\|\cdot\|_{H^1}$ are equivalent on $\mathrm{V}$.\\

We denote by $\mathrm{A}: \mathrm{D}(\mathrm{A}) \rightarrow {\mathrm{H}}$, the Stokes operator which is defined by
\begin{align*}
\mathrm{D}(\mathrm{A}) &= \mathrm{H} \cap H^2(\T),\\
 \mathrm{A} u &= - \Pi\left( \Delta u \right), ~~~u \in \mathrm{D}(\mathrm{A}).
 \end{align*}
$\mathrm{D}(\mathrm{A})$ is a Hilbert space under the graph norm,
\[|u|^2_{\mathrm{D}(\mathrm{A})} := |u|^2_\rH + |\mathrm{A} u|^2_{L^2}.\]

It is well known that $\mathrm{A}$ is a self adjoint positive operator in ${\mathrm{H}}$. Moreover
\[\mathrm{D}(\mathrm{A}^{1/2}) = \mathrm{V} \quad \mbox{and} \quad \langle Au , u \rangle_{\mathrm{H}} = \|u\|^2_{\mathrm{V}} = |\nabla u|^2_{L^2}, \,\,\, u \in \mathrm{D}(\mathrm{A}).\]

We introduce a continuous tri-linear form $b : \mathbf{L}^p \times \mathbf{W}^{1,q} \times \mathbf{L}^r \rightarrow \R$,
\[b(u,\v,w) = \sum_{i, j=1}^2 \int_{\dom} u^i \frac{\partial \v^j}{\partial x^i} w^j~dx, \quad \quad u \in \mathbf{L}(\dom)^p, \v \in \mathbf{W}^{1,q}(\dom), w \in \mathbf{L}^r(\dom)\]
where $p, q, r \in [1, \infty]$ satisfies
\[ \frac1p +\frac1q + \frac1r \leq 1.\]
By the Sobolev Embedding Theorem and the H\"older inequality, we obtain the following estimates
\begin{equation}
\label{eq:2.2}
\begin{split}
|b(u,v,w)| & \le |u|_{L^4}\|\v\|_{\mathrm{V}}|w|_{L^4}, \quad \quad u,w \in \mathbf{L}^4(\dom), \v \in \rV,\\
& \le c\|u\|_{\mathrm{V}}\|\v\|_{\mathrm{V}}\|w\|_{\mathrm{V}}, \quad \quad u,\v,w \in \rV.
\end{split}
\end{equation}

We can define a bilinear map $B : {\mathrm{V}} \times {\mathrm{V}} \rightarrow {\mathrm{V}}^{\prime}$ such that
\[\langle B(u,\v), \phi \rangle = b(u,\v, \phi),~~~~ \text{for}~u, \v, \phi \in {\mathrm{V}},\]
where $\langle \cdot, \cdot \rangle$ denotes the duality between $\rV$ and $\rV^\prime$. The following inequality is well known \cite{[Temam79]}$\colon$
\begin{equation}
\label{eq:2.3}
|b(u,\v,\phi)| \leq \sqrt{2}\, |u|^{\frac12}_{\mathrm{H}}\, \|u\|^{\frac12}_{\mathrm{V}}\,\|\v\|^{\frac12}_{\mathrm{V}}\, |\v|^{\frac12}_{\mathrm{D}(\mathrm{A})}\,|\phi|_{\mathrm{H}}, \quad u \in \mathrm{V}, \v \in \mathrm{D}(\mathrm{A}), \phi \in \mathrm{H}.
\end{equation}

Thus $b$ can be uniquely extended to the tri-linear form (denoted by the same letter)
\[b: \mathrm{V} \times \mathrm{D}(\mathrm{A}) \times \mathrm{H} \to \R.\]
We can now also extend the operator $B$ uniquely to a bounded bilinear operator
\[B : \mathrm{V} \times \mathrm{D}(\mathrm{A}) \to \mathrm{H}.\]

The following properties of the tri-linear map $b$ and the bilinear map $B$ are very well established in \cite{[BDM16], [Temam79]},
\begin{equation}
\label{eq:2.4}
\begin{split}
& b(u,u,u) = 0,~~~\quad \,\, \quad u \in \mathrm{V},\\
& b(u,w,w) = 0,~~~\quad  \quad u \in \mathrm{V}, w \in H^1,\\
& \langle B(u,u), \mathrm{A} u\rangle_{\mathrm{H}} = 0, \quad u \in \mathrm{D}(\mathrm{A}).
\end{split}
\end{equation}
We will also use the following notation, $B(u) := B(u,u)$.\\

The 2D Navier-Stokes equations driven by multiplicative Gaussian noise in the Stratonovich form  are given by:
\begin{align}
\label{eq:2.5}
\begin{cases}
du(x,t)+ \left[(u(x,t) \cdot \nabla ) u(x,t) -\nu \,\Delta u(x,t) + \nabla p(x,t)\right]dt \\
\quad \quad \quad \quad \quad \quad \quad = \sum_{j=1}^m \left[\left(c_j(x)\cdot \nabla \right) u(x,t)\right] \circ dW_j(t),\;\ t >0, \; x \in \dom, \\
\divv u(\cdot,t) = 0, \;\;t>0,\\
u(x,0) = u_0(x),\;\; x \in \dom,\\
\end{cases}
\end{align}
$u \colon [0,\infty) \times \dom \to \mathbb{R}^2$ and $p\colon [0,\infty) \times \dom  \to \mathbb{R}$ are velocity and
pressure of the fluid respectively. $\nu$ is the viscosity of the fluid (with no loss
of generality, $\nu$ will be taken  equal to $1$ for the rest of the article). Here we assume that $c_j$ are divergence free $\R^2$-valued vectors, $W_j$ are $\R-$valued i.i.d. standard Brownian motions and $\circ$ denotes the Stratonovich form. Note that the operators  $\tilde{C}_j$, $j \in \{1, \dots, m\}$, defined by $\tilde{C}_ju:= \left(c_j\cdot \nabla \right) u$, for $u \in \rV$ are skew-symmetric on $L^2(\T, \R^2)$, i.e.  $\tilde{C}_j^\ast = - \tilde{C}_j$, where $\tilde{C}_j^\ast$ denotes the adjoint of $\tilde{C}_j$ on $L^2(\T, \R^2)$.

We will be frequently using the following short-cut notation
\[Cu\circ dW(t) = \sum_{j=1}^m C_j u(t) \circ dW_j(t),\]
where $C_j = \Pi(\tilde{C}_j)$ and $\Pi$ is the Leray-Helmholtz projection operator. \\

With all the notations as defined above, the Navier-Stokes equation \eqref{eq:2.5} projected on divergence free vector field is given by
\begin{align}
\label{eq:2.7}
\begin{cases}
du(t) + \left[\mathrm{A} u(t) + B(u(t))\right]\,dt = Cu(t) \circ dW(t),\\
u(0) = u_0.
\end{cases}
\end{align}

Let us denote the set of divergence free $\mathbb{R}^2$-valued functions with unit $L^2$ norm, as following
\[\mathcal{M} = \left\{ u \in {\mathrm{H}} : |u|_{L^2} = 1\right\}.\]
Then the tangent space at $u$ is defined as,
\[T_u \mathcal{M} = \left\{\v \in {{\mathrm{H}}} : \langle \v, u \rangle_{\mathrm{H}} = 0 \right\},~~~~u \in \mathcal{M}.\]

 We define a linear map $\pi_u : {{\mathrm{H}}} \rightarrow T_u \mathcal{M}$ by
\[ \pi_u(\v) = \v - \langle \v, u \rangle_{\mathrm{H}}\, u, \]
then $\pi_u$ is the orthogonal projection from ${{\mathrm{H}}}$ into $T_u \mathcal{M}$.\\

Since for every $j \in \{1, \dots, m \}$, $C_j^\ast = - C_j$ in $\mathrm{H}$  we infer
\begin{equation}
\label{eq:2.6}
\langle C_ju, u \rangle_{\mathrm{H}} = 0, \quad u \in \mathrm{V},\;\;\;j \in \{1, \dots, m \}.
\end{equation}
In particular, if $u \in \rV \cap \mathcal{M}$, then $C_ju \in T_u\mathcal{M}$ for every $j \in \{1, \dots, m \}$ and hence won't produce any correction terms when projected on the tangent space $T_u \mathcal{M}$, which is shown explicitly below.\\

Let
\[F(u) = \mathrm{A} u + B(u,u) - Cu \circ dW(t)\]
and $\hat{F}(u)$ be the projection of $F(u)$ onto the tangent space $T_u\mathcal{M}$, then
\begin{align*}
\hat{F}(u) &= \pi_u(F(u)) = F(u) - \langle F(u), u \rangle_{\mathrm{H}} \, u \\
&= \mathrm{A} u + B(u) - Cu \circ dW - \langle Au + B(u) - Cu \circ dW, u \rangle_{\mathrm{H}} \, u \\
&= \mathrm{A} u - \langle \mathrm{A} u, u \rangle_{\mathrm{H}} \, u + B(u) - \langle B(u), u \rangle_{\mathrm{H}} \, u  - Cu \circ dW + \langle Cu , u \rangle_{\mathrm{H}} u \circ dW\\
&= \mathrm{A} u - |\nabla u|_{L^2}^2\,u + B(u) - Cu \circ dW.
\end{align*}
The last equality follows from \eqref{eq:2.6} and the identity that $\langle B(u), u\rangle_{\mathrm{H}} = 0$.

Thus by projecting NSEs \eqref{eq:2.7} onto the tangent space $T_u\mathcal{M}$, we obtain the following Stochastic Constrained Navier-Stokes Equations (SCNSEs)
\begin{align}
\label{eq:2.8}
\begin{cases}
du(t) + \left[\mathrm{A} u(t) + B(u(t))\right]dt = |\nabla u(t)|_{L^2}^2u(t)\,dt + Cu(t) \circ dW(t),\\
u(0) = u_0 \in \mathrm{V} \cap \mathcal{M}.
\end{cases}
\end{align}

\section{Stochastic Constrained Navier-Stokes equations}
\label{s:4}

We consider the following stochastic evolution equation
\begin{equation}
\label{eq:4.1}
\begin{cases}
&du(t) + \left[\mathrm{A} u(t) + B(u(t))\right]dt = |\nabla u(t)|_{L^2}^2u(t)\,dt + Cu(t) \circ dW(t), \quad t \in [0,T],\\
&u(0) = u_0,
\end{cases}
\end{equation}
where $Cu(t,x) \circ dW(t) := \sum_{j=1}^m C_j u(t,x) \circ dW_j(t)$ with   $C_j u =  \Pi\left( (c_j \cdot \nabla ) u\right)$ and
$W_j, j= 1,\dots, m,$ are  i.i.d standard $\R-$valued Brownian Motions.\\

From now on we will assume that $c_j$ are constant vector fields. Whether our results are true in a more general setting is an open problem.\\
\noindent \textbf{Assumptions.} We assume that

\begin{trivlist}
\item{(A.1)} Vectors $c_1, \ldots, c_m$ belong to $\R^2$ such that $K_c^2 < 1$, where
\begin{equation}
\label{eqn-K_c}
K_c:=\mathrm{max}_{j \in \{1, \cdots, m\}} |c_j|_{\R^2}\,,
\end{equation}
$|\cdot|_{\R^2}$ is the Euclidean norm in $\R^2$.

\item{(A.2)} $u_0 \in \mathrm{V} \cap \mathcal{M}$.
\end{trivlist}

\begin{definition}
\label{defn_stoc_basis}
A stochastic basis $(\Omega, \mathcal{F}, \mathbb{F}, \mathbb{P})$ is a probability space equipped with the filtration $\mathbb{F} = \{\mathcal{F}_t\}_{t \ge 0}$ of its $\sigma-$field $\mathcal{F}$.
\end{definition}

\begin{definition}
\label{defn4.2}
We say that problem \eqref{eq:4.1} has a \textbf{strong solution} iff for every stochastic basis $(\Omega, \mathcal{F}, \mathbb{F}, \mathbb{P})$ and every $\R^m-$ valued $\mathbb{F}-$Wiener process $W = \left(W(t)\right)_{t \ge 0}$, there exists a $\mathbb{F}-$progressively measurable process $u : [0,T] \times {\Omega} \to \mathrm{D(\mathrm{A})}$ with ${\mathbb{P}}$-a.e. paths
\[u(\cdot, \omega) \in \ccal([0,T]; \rV) \cap L^2(0,T; \mathrm{D}(\mathrm{A})),\]
such that for all $t \in [0,T]$ and all $\v \in \rV$ ${\mathbb{P}}$-a.s.
\begin{equation}
\label{eq:4.2}
\begin{split}
&\langle u(t), \v \rangle - \langle u_0, \v \rangle + \int_0^t \langle \mathrm{A} u(s), \v \rangle\,ds + \int_0^t \langle B(u(s)), \v \rangle\,ds\\
&=  \int_0^t |\nabla u(s)|_{L^2}^2 \langle u(s), \v \rangle\,ds + \frac12\int_0^t \sum_{j=1}^m \langle C_j^2 u(s), \v \rangle\,ds + \int_0^t \sum_{j=1}^m  \langle C_ju(s),\v\rangle\,d\hat{W}_j(s).
\end{split}
\end{equation}
\end{definition}

\begin{definition}
\label{defn4.1}
We say that there exists a \textbf{martingale solution} of \eqref{eq:4.1} iff there exist
\begin{itemize}
\item a stochastic basis $(\hat{\Omega}, \hat{\mathcal{F}}, \hat{\mathbb{F}}, \hat{\mathbb{P}})$,

\item an $\R^m-$valued $\hat{\mathbb{F}}-$Wiener process $\hat{W}$,

\item and a $\hat{\mathbb{F}}-$progressively measurable process $u : [0,T] \times \hat{\Omega} \to \mathrm{D(\mathrm{A})}$ with $\hat{\mathbb{P}}$-a.e. paths
\[u(\cdot, \omega) \in \ccal([0,T]; \rV_{\mathrm{w}}) \cap L^2(0,T; \mathrm{D}(\mathrm{A})),\]
such that for all $t \in [0,T]$ and all $\v \in \rV$
the identity \eqref{eq:4.2} holds $\hat{\mathbb{P}}$-a.s.
\end{itemize}
\end{definition}

Next we state some important results of this paper which will be proved in further sections.

\begin{theorem}
\label{thm5.1}
Let assumptions $(A.1)-(A.2)$ be satisfied. Then there exists a martingale solution $(\hat{\Omega}, \hat{\mathcal{F}}, \hat{\mathbb{F}}, \hat{\mathbb{P}}, \hat{W}, u)$ of problem \eqref{eq:4.1} such that
\begin{equation}
\label{eq:5.1}
\hat{\E} \left[\sup_{t \in [0,T]}\|u(t)\|^2_{\mathrm{V}} + \int_0^T |u(t)|^2_{\mathrm{D}(\mathrm{A})}\,dt \right] < \infty.
\end{equation}
\end{theorem}

\begin{remark}
\label{rem3.5}
The solution obtained in the above theorem is weak in probabilistic sense and strong in PDE sense.
\end{remark}

The next lemma shows that almost all the trajectories of the solution obtained in Theorem~\ref{thm5.1} are almost everywhere equal to a continuous $\rV$-valued function defined on $[0,T]$.

\begin{lemma}
\label{lemma6.1}
Assume that the assumptions $(A.1)-(A.2)$ are satisfied. Let $(\ho, \hat{\mathcal{F}}, \hat{\mathbb{F}}, \hp, \hat{W}, u)$ be a martingale solution of \eqref{eq:4.1} such that
\begin{equation}
\label{eq:6.1}
\hat{\E} \left[ \sup_{t \in [0,T] }\|u(t)\|^2_\rV + \int_0^T |u(s)|^2_{\rD(\mathrm{A})}\,ds \right] < \infty.
\end{equation}
Then for $\hp$ almost all $\omega \in \hat{\Omega}$ the trajectory $u(\cdot, \omega)$ is almost everywhere equal to a continuous $\rV-$valued function defined on $[0,T]$. Moreover for every $t \in [0,T], \hp-$a.s.
\begin{align}
\label{eq:6.2}
u(t) & = u_0 - \int_0^t \left[ \mathrm{A} u(s) + B(u(s)) - |\nabla u(s)|_{L^2}^2\,u(s) \right]\,ds \nonumber \\
&~~~+ \frac12\int_0^t \sum_{j=1}^m C^2_j u(s)\,ds + \int_0^t \sum_{j=1}^m C_ju(s)\,d\hat{W}(s).
\end{align}
\end{lemma}

\begin{definition}
\label{defn6.2}
Let $(\Omega, \mathcal{F}, \mathbb{F}, \mathbb{P}, W, u^i)$, $i = 1,2$ be the martingale solutions of \eqref{eq:4.1} with $u^i(0) = u_0$, $i= 1,2$. Then we say that the solutions are \textbf{pathwise unique} if $\mathbb{P}-$a.s. for all $t \in [0,T]$, $u^1(t) = u^2(t)$.
\end{definition}
In Lemma \ref{lemma6.3} we will show that the pathwise uniqueness property for our problem holds. This will enable us to deduce
the following theorem that summarises the main result of our paper$\colon$
\begin{theorem}
\label{thm6.6}
For every $u_0 \in \rV$ there exists a pathwise unique strong solution $u$ of stochastic constrained Navier-Stokes equation \eqref{eq:4.1} such that
\begin{equation}
\label{eq:6.8}
\E \left[ \int_0^T |u(t)|^2_{\mathrm{D}(\mathrm{A})}\,dt + \sup_{t \in [0,T]} \|u(t)\|^2_\rV \right] < \infty.
\end{equation}
\end{theorem}

\begin{remark}
\label{rem3.9}
The solution of \eqref{eq:4.1} obtained in previous theorem is strong in both probabilistic and PDE sense.
\end{remark}

\section{Compactness}
\label{s:3}
Let us consider the following functional spaces:

\noindent $\ccal([0,T]; \mathrm{H}) :=$ the space of continuous functions $u: [0,T] \to \mathrm{H}$ with the topology $\mathcal{T}_1$ induced by the norm $|u|_{\ccal([0,T]; \mathrm{H})} := \sup_{t \in [0,T]}|u(t)|_{\mathrm{H}}$,\\

\noindent $L^2_{\mathrm{w}}(0,T; \mathrm{D}(\mathrm{A})) :=$ the space $L^2(0,T; \mathrm{D}(\mathrm{A}))$ with the weak topology $\mathcal{T}_2$,\\

\noindent $L^2(0,T; \mathrm{V}) :=$ the space of measurable functions $u : [0,T] \to \mathrm{V}$ such that
\[|u|_{L^2(0,T; \mathrm{V})} = \left(\int_0^T \|u(t)\|^2_\rV \,dt \right)^{\frac12} < \infty,\]
with the topology $\mathcal{T}_3$ induced by the norm $|u|_{L^2(0,T; \mathrm{V})}$.\\
\noindent Let $\rV_{\mathrm{w}}$ denote the Hilbert space $\rV$ endowed with the weak topology.\\
\noindent $\ccal([0,T]; \mathrm{V}_{\mathrm{w}}) :=$ the space of weakly continuous functions $u: [0,T] \to \mathrm{V}$ endowed with the weakest topology $\mathcal{T}_4$ such that for all $h \in \mathrm{V}$ the mappings
\[\ccal([0,T]; \mathrm{V}_{\mathrm{w}}) \ni u \to \langle u(\cdot), h \rangle_{\mathrm{V}} \in \ccal([0,T]; \R)\]
are continuous. In particular, $u_n \to u$ in $\ccal([0,T]; \rV_{\mathrm{w}})$ iff for all $h \in \rV \colon$
\[\lim_{n \to \infty} \sup_{t \in [0,T]} \left|\langle u_n(t) - u(t), h \rangle_\rV \right| = 0.\]

Consider the ball
\[\mathbb{B} := \{ x \in \rV : \|x\|_{\rV} \le r \}.\]

Let $q$ be the metric compatible with the weak topology on $\mathbb{B}$. Let us consider the following subspace of the space $\ccal([0,T]; \rV_{\mathrm{w}})$
\begin{align}
\label{eq:5.4.1}
\ccal([0,T]; \mathbb{B}_{\mathrm{w}}) =\,\, & \text{the space of weakly continuous functions } \; u \colon [0,T] \to \rV\nonumber \\
& \text{such that } \sup_{t \in [0,T]}\|u(t)\|_\rV \le r.
\end{align}
The space $\ccal([0,T]; \mathbb{B}_{\mathrm{w}})$ is metrizable (see \cite{[Brezis83], [BM13]}) with metric
\begin{equation}
\label{eq:5.4.2}
\varrho(u,\v) = \sup_{t \in [0,T]} q(u(t),\v(t)).
\end{equation}

Since by the Banach-Alaoglu theorem $\mathbb{B}_{\mathrm{w}}$ is compact, $(\ccal([0,T]; \mathbb{B}_{\mathrm{w}}), \varrho)$ is a complete metric space.

The following lemma \cite[Lemma~2.1]{[BM14]} says that any sequence $(u_n)_{n \in \N} \subset \ccal([0,T]; \mathbb{B})$ convergent in $\ccal([0,T]; \rH)$ is also convergent in the space $\ccal([0,T]; \mathbb{B}_{\mathrm{w}})$.

\begin{lemma}
\label{lemma5.4.1}
Let $u_n \colon [0,T] \to \rV, n \in \N$ be functions such that
\begin{itemize}
\item[(i)] $\sup_{n \in \N} \sup_{s \in [0,T]} \|u_n(s)\|_\rV \le r$,

\item[(ii)] $u_n \to u$ in $\ccal([0,T]; \rH)$.
\end{itemize}
Then $u,u_n \in \ccal([0,T]; \mathbb{B}_{\mathrm{w}})$ and $u_n \to u$ in $\ccal ([0,T]; \mathbb{B}_{\mathrm{w}})$ as $n \to \infty$.
\end{lemma}

Let
\begin{equation}
\label{eq:3.1}
\mathcal{Z}_T = \ccal([0,T]; \mathrm{H}) \cap L^2_{\mathrm{w}}(0,T; \mathrm{D}(\mathrm{A})) \cap L^2(0,T; \mathrm{V}) \cap \ccal(0,T; \mathrm{V}_{\mathrm{w}}),
\end{equation}
and let $\mathcal{T}$ be the supremum of the corresponding topologies.\\

Now we formulate the compactness criterion analogous to the result due to Mikulevicus and Rozowskii \cite{[MR05]}, Brze\'zniak and Motyl \cite{[BM14]} for the space $\mathcal{Z}_T$ .

\begin{lemma}
\label{lemma5.4.2}
Let $\mathcal{Z}_T$, $\mathcal{T}$ be as defined in \eqref{eq:3.1}. Then a set $\kcal \subset \mathcal{Z}_T$ is $\mathcal{T}-$relatively compact if the following three conditions hold
\begin{itemize}
\item[(\rm{a})] $\sup_{u \in \kcal} \sup_{s \in [0,T]} \|u(s)\|_\rV < \infty\,,$
\item[(\rm{b})] $\sup_{u \in \kcal} \int_0^T |u(s)|^2_{\mathrm{D}(\rA)}\,ds < \infty\,$, i.e. $\kcal$ is bounded in $L^2(0,T; \mathrm{D}(\rA))$,
\item[(\rm{c})] $\lim_{\delta \to 0} \sup_{u \in \kcal} \sup_{\underset{|t-s| \le \delta}{s,t \in [0,T]}}|u(t) - u(s)|_\rH = 0\,.$
\end{itemize}
\end{lemma}

\begin{proof}
Let $\kcal$ be a subset of $\mathcal{Z}_T$. Because of the assumption (\rm{a}) we may consider the metric space $\ccal([0,T]; \mathbb{B}_{\mathrm{w}}) \subset \ccal([0,T]; \rV_{\mathrm{w}})$ defined by \eqref{eq:5.4.1} and \eqref{eq:5.4.2} with $r = \sup_{u \in \kcal}\sup_{s \in [0,T]}\|u(s)\|_\rV$. Because of the assumption (\rm{b}) the restriction to $\kcal$ of the weak topology in $L^2(0,T; \mathrm{D}(\rA))$ is metrizable. Since the restrictions to $\kcal$ of the four topologies considered in $\mathcal{Z}_T$ are metrizable, compactness of a subset of $\mathcal{Z}_T$ is equivalent to its sequential compactness.

Let $(u_n)$ be a sequence in $\kcal$. By the Banach-Alaoglu Theorem, condition (\rm{b}) yields that $\bar{\kcal}$ is compact in $L^2_{\mathrm{w}}(0,T; \mathrm{D}(\rA))$. Condition (\rm{c}) implies that the functions $u_n$ are equicontinuous in $\ccal([0,T], \rH)$. Since the embeddings  $\mathrm{D}(\rA) \hookrightarrow \rV \hookrightarrow \rH$ are continuous and the embedding $\mathrm{D}(\rA) \hookrightarrow \rV$ is compact, then Dubinsky Theorem (see \cite[Theorem~IV.4.1]{[VF88]}) with conditions (\rm{b}) and (\rm{c}) imply that $\kcal$ is compact in $L^2(0,T; \rV) \cap \ccal([0,T]; \rH)$. Hence in particular, there exists a subsequence, still denoted by $(u_n)$, convergent in $\rH$. Therefore by Lemma~\ref{lemma5.4.1} $(u_n)$ is convergent in $\ccal([0,T]; \mathbb{B}_{\mathrm{w}})$. This completes the proof of the lemma.
\end{proof}

\subsection{Tightness}
\label{s:3.1}
Let $(\mathbb{S}, \varrho)$ be a separable and complete metric space.

\begin{definition}
\label{defn3.1}
Let $u \in \ccal([0,T]; \mathbb{S})$. The modulus of continuity of $u$ on $[0,T]$ is defined by
\[m(u, \delta) :=  \sup_{s,t \in [0,T],\,|t - s|\le \delta} \varrho (u(t), u(s)), \quad \delta > 0.\]
\end{definition}

Let $(\Omega, \mathcal{F}, \mathbb{P})$ be a probability space with filtration $\mathbb{F}:= (\mathcal{F}_t)_{t \in [0,T]}$ satisfying the usual conditions, see \cite{[Metivier82]}, and let $(X_n)_{n \in \mathbb{N}}$ be a sequence of continuous $\mathbb{F}$-adapted $\mathbb{S}$-valued processes.

\begin{definition}
\label{defn3.2}
We say that the sequence $(X_n)_{n \in \mathbb{N}}$ of $\mathbb{S}$-valued random variables satisfies condition $[\mathbf{T}]$ iff $\forall\, \varepsilon >0, \forall\, \eta > 0,\, \exists\, \delta > 0$:
\begin{equation}
\label{eq:3.2}
\sup_{n \in \mathbb{N}} \mathbb{P}\left\{m(X_n, \delta) > \eta\right\} \le \varepsilon.
\end{equation}
\end{definition}

\begin{lemma}
\label{lemma3.3}
Assume that $(X_n)_{n \in \mathbb{N}}$ satisfies condition $[\mathbf{T}]$. Let $\mathbb{P}_n$ be the law of $X_n$ on $\ccal([0,T]; \mathbb{S})$, $n \in \mathbb{N}$. Then for every $\varepsilon > 0$ there exists a subset $A_\varepsilon \subset \ccal([0,T]; \mathbb{S})$ such that
\[\sup_{n \in \mathbb{N}} \mathbb{P}_n(A_\varepsilon) \ge 1 - \varepsilon\]
and
\begin{equation}
\label{eq:3.3}
\lim_{\delta \to 0} \sup_{u \in A_\varepsilon} m(u, \delta) = 0.
\end{equation}
\end{lemma}

Now we recall the Aldous condition $[\mathbf{A}]$, which is connected with condition $[\mathbf{T}]$. This condition allows to investigate the modulus of continuity for the sequence of stochastic processes by means of stopped processes.

\begin{definition}
\label{defn3.4}
A sequence $(X_n)_{n \in \mathbb{N}}$ satisfies condition $[\mathbf{A}]$ iff $\forall\, \varepsilon > 0$, $\forall\, \eta > 0$, $\exists \, \delta > 0$ such that for every sequence $(\tau_n)_{n \in \mathbb{N}}$ of $\mathbb{F}$-stopping times with $\tau_n \le T$ one has
\[\sup_{n \in \mathbb{N}} \sup_{0 \le \theta \le \delta} \mathbb{P}\left\{\varrho(X_n(\tau_n + \theta), X_n(\tau_n)) \ge \eta \right\} \le \varepsilon.\]
\end{definition}

\begin{lemma}
\label{lemma3.5}
Conditions $[\mathbf{A}]$ and $[\mathbf{T}]$ are equivalent.
\end{lemma}

Using the compactness criterion from Lemma~\ref{lemma5.4.2} and above results corresponding to Aldous condition we obtain the following corollary which we will use to prove the tightness of the laws defined by the Galerkin approximations.

\begin{corollary}[Tightness criterion]
\label{cor3.6}
Let $(X_n)_{n \in \mathbb{N}}$ be a sequence of continuous $\mathbb{F}$-adapted $\mathrm{H}$-valued processes such that
\begin{trivlist}
\item{(a)} there exists a constant $C_1 > 0$ such that
\[\sup_{n \in \mathbb{N}} \E \left[ \sup_{s \in [0,T]} \|X_n(s)\|^2_{\mathrm{V}} \right] \le C_1,\]

\item{(b)} there exists a constant $C_2 > 0$ such that
\[\sup_{n \in \mathbb{N}} \E \left[ \int_0^T |X_n(s)|^2_{\mathrm{D}(\mathrm{A})}\,ds \right] \le C_2,\]

\item{(c)} $(X_n)_{n \in \mathbb{N}}$ satisfies the Aldous condition $[\mathbf{A}]$ in $\mathrm{H}$.
\end{trivlist}
Let $\tilde{\mathbb{P}}_n$ be the law of $X_n$ on $\mathcal{Z}_T$. Then for every $\varepsilon > 0$ there exists a compact subset $K_\varepsilon$ of $\mathcal{Z}_T$ such that
\[ \sup_{n \in \mathbb{N}} \tilde{\mathbb{P}}_n(K_\varepsilon) \ge 1 - \varepsilon.\]
\end{corollary}

\begin{proof}
Let $\varepsilon > 0$. By the  Chebyshev inequality and $(a)$, we infer that for any $n \in \N $ and any $r>0$
\[
  \tilde{\mathbb{P}}_n \biggl( \sup_{s \in [0,T]} \|X_n (s){\|}_{\rV}^{2} > r  \biggr)
  \le \frac{ \tilde{\E}_n \bigl[ \sup_{s \in [0,T]} \|X_n (s) {\|}_{\rV}^{2} \bigr]}{r}
  \le \frac{{C}_{1}}{r}.
\]
Let ${R}_{1}$ be such that $\frac{{C}_{1}}{{R}_{1}} \le \frac{\varepsilon }{3}$. Then
\[
  \sup_{n\in \N } \tilde{\mathbb{P}}_n \biggl( \sup_{s \in [0,T]}\|X_n (s){\|}_{\rV}^{2} > {R}_{1} \biggr) \le \frac{\varepsilon }{3}.
\]
Let ${B}_{1}:= \left\{ u \in \zcal_T :\, \, \sup_{s \in [0,T]}\|u(s) {\|}_{\rV}^{2} \le {R}_{1} \right\} $.\\
By the  Chebyshev inequality and $(b)$, we infer that for any $n \in \N $ and any $r>0$
\[
\tilde{\mathbb{P}}_n \bigl( |X_n|_{{L}^{2}(0,T; \rD(\rA))} > r  \bigr)
  \le \frac{\tilde{\E}_n \bigl[ |X_n|_{{L}^{2}(0,T; \rD(\rA))}^2 \bigr]  }{{r}^{2}}
  \le \frac{{C}_{2}}{{r}^{2}}.
\]
Let ${R}_{2}$ be such that $\frac{{C}_{2}}{{R}_{2}^{2}} \le \frac{\varepsilon }{3}$. Then
\[
  \sup_{n\in \N }  \tilde{\mathbb{P}}_n \bigl( |X_n |_{{L}^{2}(0,T; \rD(\rA))} > {R}_{2}  \bigr) \le \frac{\varepsilon }{3}.
\]
Let ${B}_{2} := \left\{ u \in \zcal_T : \, \, |u|_{{L}^{2}(0,T; \rD(\rA))} \le {R}_{2} \right\} $.\\
By Lemmas \ref{lemma3.3} and \ref{lemma3.5} there exists a subset
${A}_{\frac{\varepsilon }{3}} \subset \ccal ([0,T], \rH) $ such that
\[{\tilde{\mathbb{P} }}_{n} \bigl( {A}_{\frac{\varepsilon }{3}}\bigr) \ge 1 - \frac{\varepsilon }{3}\] 
and
\[
   \lim_{\delta \to 0 }  \sup_{u \in {A}_{\frac{\varepsilon }{3}}}
\sup_{\underset{|t-s| \le \delta }{s,t \in [0,T]}}  |u(t) - u(s){|}_{\rH} =0 .
\]
It is sufficient to define ${K}_{\varepsilon } $ as the closure  of the set ${B}_{1} \cap {B}_{2} \cap {A}_{\frac{\varepsilon }{3}}$ in $\zcal_T$. By Lemma  \ref{lemma5.4.2}, ${K}_{\varepsilon }$ is compact in $\zcal_T$. The proof is thus complete.
\end{proof}

\subsection{The Skorokhod Theorem}
\label{s:3.2}

We will use the following Jakubowski's generalisation of the Skorokhod Theorem in the form given by Brze\'{z}niak and Ondrej\'{a}t \cite{[BO11]}, see also \cite{[Jakubowski97]}.

\begin{theorem}
\label{thm3.7}
Let $\mathcal{X}$ be a topological space such that there exists a sequence $\{f_m\}_{m \in \mathbb{N}}$ of continuous functions $f_m : \mathcal{X} \to \R$ that separates points of $\mathcal{X}$. Let us denote by $\mathcal{S}$ the $\sigma$-algebra generated by the maps $\{f_m\}$. Then
\begin{trivlist}
\item{(a)} every compact subset of $\mathcal{X}$ is metrizable,
\item{(b)} if $(\mu_m)_{m \in \mathbb{N}}$ is a tight sequence of probability measures on $(\mathcal{X}, \mathcal{S})$, then there exists a subsequence $(m_k)_{k \in \mathbb{N}}$, a probability space $(\Omega, \mathcal{F}, \mathbb{P})$ with $\mathcal{X}$-valued Borel measurable variables $\xi_k, \xi$ such that $\mu_{m_k}$ is the law of $\xi_k$ and $\xi_k$ converges to $\xi$ almost surely on $\Omega$. Moreover, the law of $\xi$ is a Radon measure.
\end{trivlist}
\end{theorem}

\begin{lemma}\label{lemma5.4.5} The topological space ${\mathcal{Z}}_{T}$  satisfies the assumptions of Theorem~\ref{thm3.7}.
\end{lemma}

\begin{proof}
We want to prove that on each space appearing in the definition \eqref{eq:3.1} of the space $\mathcal{Z}_T$ there exists a countable set of continuous real-valued functions separating points.

Since the spaces $\ccal ([0,T]; \rH)$ and ${L}^{2}(0,T; \rV)$ are separable, metrizable and complete, this condition is satisfied, see \cite{[Badrikian70]}, expos\'{e} 8.

For the space ${L}^{2}_{\mathrm{w}}(0,T; \rD(\rA))$ it is sufficient to put
\[
    {f}_{m}(u):= \int_{0}^{T} \langle u(t), \v_m(t) \rangle_{\rD(\rA)} \, dt \in \R ,
 \qquad u \in {L}^{2}_{\mathrm{w}}(0,T; \rD(\rA)),\quad m \in \N ,
\]
where $\{ {\v}_{m}, m \in \N  \} $ is a dense subset of ${L}^{2}(0,T; \rD(\rA))$.

Let us consider the space $\ccal ([0,T];{\rV}_{\mathrm{w}})$. Let $\{ {h}_{m}, \, m \in \N  \}  $ be any dense subset of $\rH$ and let ${\mathbb{Q}}_{T}$ be the  set of rational numbers belonging to the interval $[0,T]$.
Then the family $\{ {f}_{m,t}, \, m \in \N , \, \, t \in {\mathbb{Q}}_{T} \} $ defined by
\[
      {f}_{m,t}(u):= \langle u(t), h_m \rangle_\rV \in \R ,
 \qquad u \in  \ccal ([0,T];{\rV}_{\mathrm{w}}), \quad m \in \N ,
      \quad t \in {\mathbb{Q}}_{T}
\]
consists of continuous functions separating points in $\ccal ([0,T];{\rV}_{\mathrm{w}})$, thus concluding the proof of the lemma.
\end{proof}

Using Theorem~\ref{thm3.7} and Lemma~\ref{lemma5.4.5}, we obtain the following corollary which we will apply to construct a martingale solution to the stochastic constrained Navier-Stokes equations \eqref{eq:4.1}.

\begin{corollary}
\label{cor3.8}
Let $(\eta_n)_{n \in \mathbb{N}}$ be a sequence of $\mathcal{Z}_T$-valued random variables such that their laws $\mathcal{L}(\eta_n)$ on $(\mathcal{Z}_T, \mathcal{T})$ form a tight sequence of probability measures. Then there exists a subsequence $(n_k)$, a probability space $(\tilde{\Omega}, \tilde{\mathcal{F}}, \tilde{\mathbb{P}})$ and $\mathcal{Z}_T$-valued random variables $\tilde{\eta}$, $\tilde{\eta}_k, k \in \mathbb{N}$ such that the variables $\eta_k$ and $\tilde{\eta}_k$ have the same laws on $\mathcal{Z}_T$ and $\tilde{\eta}_k$ converges to $\tilde{\eta}$ almost surely on $\tilde{\Omega}$.
\end{corollary}

\section{Faedo-Galerkin approximation and existence of a martingale solutions}
\label{s:5}
As mentioned in the introduction, the proof of the existence of a martingale solution is based on the Faedo-Galerkin approximation. In this subsection we first talk about the basic ingredients required for the approximation and then obtain the a'priori estimates, which we later use in the Subsection~\ref{s:5.2} to prove the tightness of laws induced by the solutions of the approximating equations \eqref{eq:5.3}.\\

Let $\{e_i\}_{i=1}^\infty$ be the orthonormal basis in $\mathrm{H}$ composed of eigenvectors of $\mathrm{A}$. Let 
\[\mathrm{H}_n := \mathrm{span} \{e_1, \dots, e_n\}\]
be the subspace with the norm inherited from $\mathrm{H}$, then $P_n : \mathrm{H} \to \mathrm{H}_n$ given by
\begin{equation}
\label{eq:5.2}
P_n u := \sum_{i=1}^n \langle u, e_i \rangle_{\mathrm{H}}\, e_i\,, \quad u \in \mathrm{H}\,,
\end{equation}
is the orthogonal projection onto $\rH_n$.\\

\noindent Let us consider the classical Faedo-Galerkin approximation of \eqref{eq:4.1} in the space $\mathrm{H}_n \colon$

\begin{equation}
\label{eq:5.3}
\begin{cases}
du_n(t) = - \left[ P_n \mathrm{A}u_n(t) + P_n B(u_n(t)) +|\nabla u_n(t)|_{L^2}^2u_n(t) \right]dt \\
\quad \quad \quad \quad + \sum_{j=1}^m P_n C_j u_n(t) \circ dW_j(t), \quad \quad \quad \;\;\;\;\;\;\; t \in [0,T],\\
u_n(0) = \frac{P_n u_0}{|P_n u_0|}.
\end{cases}
\end{equation}

Using the idea from \cite{[Hussain15]} and the Banach Fixed Point Theorem we can show that the SDE \eqref{eq:5.3} has a local maximal solution up to some stopping time $\tau\leq T$. In the following lemma we show that this local solution stays on the manifold $\mathcal{M}$, i.e. $u_n(t) \in \mathcal{M}$ for every $t \in [0,\tau)$.
\begin{lemma}
\label{lemma5.3}
Let $u_0 \in \mathrm{V} \cap \mathcal{M}$ then the solution of \eqref{eq:5.3} stays on the manifold $\mathcal{M}$, i.e. for all $t \in [0,\tau)$, $u_n(t) \in \mathcal{M}$.
\end{lemma}

\begin{proof}
Let $u_n$ be the solution of \eqref{eq:5.3}. Then applying It\^o formula to the function $|x|^2_\rH$ and the process $u_n$ along \eqref{eq:5.3}, \eqref{eq:2.4} and assumption $(A.1)$, we get

\begin{align*}
\frac12 d|u_n(t)|_\rH^2 & =  \langle u_n(t), - P_n \mathrm{A} u_n(t) - P_nB(u_n(t)) + |\nabla u_n(t)|_{L^2}^2 u_n(t)\rangle_\rH\,dt  \\
&~~ + \dfrac{1}{2} \sum_{j=1}^m \langle u_n(t),  (P_n C_j)^2 u_n(t) \rangle_\rH\, dt + \dfrac{1}{2} \sum_{j=1}^m \langle P_n C_j u_n(t), P_n C_j u_n(t) \rangle_\rH\,dt  \\
&~~ + \sum_{j=1}^m \langle u_n(t), P_n C_j u_n(t)\,dW_j(t) \rangle_\rH \\
& = - \|u_n(t)\|_{\mathrm{V}}^2 dt +  |\nabla u_n(t)|_{L^2}^2 |u_n(t)|_\rH^2 dt + \frac12 \sum_{j=1}^m \langle C_j^{\ast} u_n(t), C_j u_n(t) \rangle_{\rH}\,dt \\
&~~ + \frac12 \sum_{j=1}^m |C_j u_n(t)|_\rH^2\,dt \\
& = \|u_n(t)\|_{\mathrm{V}}^2 \left[ |u_n(t)|^2_\rH - 1 \right] dt + \frac12 \sum_{j=1}^m \left[|C_j u_n(t)|_\rH^2 - |C_j u_n(t)|_\rH^2 \right] dt
\end{align*}
thus we get,
\[ d \left[|u_n(t)|^2_\rH - 1 \right] = 2\|u_n(t)\|_{\mathrm{V}}^2 \left[|u_n(t)|_\rH^2 - 1 \right] dt. \]
Integrating on both sides from $0$ to $t$, we obtain
\[ |u_n(t)|^2 - 1 = \left[|u_n(0)|^2_\rH - 1 \right] \exp{ \left[ 2 \int_0^t \|u_n(s)\|_{\mathrm{V}}^2~ds \right]}. \]
Now since $|u_n(0)|_\rH = 1$ and $\int_0^t \|u_n(s)\|_{\mathrm{V}}^2\,ds < \infty$, we get $|u_n(t)|_\rH = 1$ for all $t \in [0, \tau)$, i.e $u_n(t) \in \mathcal{M}$ for every $t \in [0, \tau)$.
\end{proof}

Since on the finite dimensional space $\rH_n$ the $\rH$ and $\rV$ norm are equivalent, we can infer from the previous lemma that the $\rV$-norm of the solution stays bounded. Hence using this non-explosion result as in the case of deterministic setting \cite{[BDM16]} we can prove the following lemma$\colon$

\begin{lemma}
\label{lemma5.2}
For each $n \in \mathbb{N}$, there exists a global solution of \eqref{eq:5.3}. Moreover for every $T > 0$, $u_n \in \ccal([0,T]; \mathrm{H}_n), \mathbb{P}$-a.s. and for any $q \in [2, \infty)$
\[ \E \left[\int_0^T |u_n(s)|^q_{\mathrm{H}}\,ds \right] < \infty.\]
\end{lemma}

\subsection{A'priori estimates}
\label{s:5.1}

We will require the following lemma to obtain a'priori bounds.

\begin{lemma}
\label{lemma5.4}
Let $c \in \R^2$ and let $\mathbf{c}:\mathbb{T}^2\to \mathbb{R}^2$  be the corresponding  constant vector field. Put, for $u \in H^{1,2}(\mathbb{T}^2, \mathbb{R}^2)$,
\[\tilde{C}u = \mathbf{c} \cdot \nabla\, u \quad \quad \mbox{and}\quad \quad C u = \Pi(\tilde{C}u).\]
If the vector field $u\in H^{2,2}(\mathbb{T}^2, \mathbb{R}^2)$ is divergence free, then $\tilde{C} u$ is divergence free as well. Moreover, 
\begin{equation}
\label{eq:5.4}
\rA Cu - C \rA u = 0\,,\quad\;\; u \in H^{3,2}(\mathbb{T}^2, \mathbb{R}^2).
\end{equation}
\end{lemma}

\begin{proof}
Let $c = (c_1, c_2)$ then $\tilde{C}  u  = (c_1 D_1 + c_2 D_2) u$. We have
\begin{align*}
\divv(\tilde{C} u) & = D_1\bigl((c_1 D_1 + c_2 D_2) u_1 \bigr) + D_2 \bigl( (c_1 D_1 + c_2 D_2 ) u_2 \bigr)\\
\delb{& = D_1 c_1 D_1 u_1 + c_1 D_1^2 u_1 + D_1 c_2 D_2 u_1 + c_2 D_{12} u_1\\
&\quad + D_2 c_1 D_1 u_2 + c_1 D_{12} u_2 + D_2 c_2 D_2 u_2 + c_2 D_2^2 u_2 \\}
& = c_1 D_1D_1 u_1 + c_2 D_1D_2 u_1 + c_1 D_{2}D_{1} u_2  + c_2 D_2D_{2} u_2 \\
& =  c_1 D_1\bigl(D_1 u_1 +  D_2 u_2 \bigr) + c_2 D_2\bigl( D_1 u_1+ D_2 u_2  \bigr) \\
& = \left(c_1D_1 + c_2 D_2 \right)\left(\divv u\right) = 0\,,
\end{align*}
where  we  used that vector $c$ is constant and $u$ is divergence free respectively. In order to establish the equality \eqref{eq:5.4} we start by considering $\mathrm{A} C u - C \mathrm{A} u$. Since $\rA u$ is divergence free, from the previous calculations we have $\Pi(\tilde{C} \rA u) = \tilde{C} \rA u$. Thus
\begin{align*}
\rA C u - C \rA u & = - \Delta \bigl( (c_1 D_1 + c_2 D_2)u\bigr) - \bigl(c_1 D_1 + c_2 D_2\bigr)\left(-\Delta u \right) \\
& = - \left[\delb{\Delta(c_1)D_1 u +} c_1 \Delta D_1 u \delb{+ \Delta(c_2) D_2 u} + c_2 \Delta D_2 u \delb{+ 2 D_1 c_1 D_1^2 u + 2 D_2 c_2 D^2_2 u} \right] + \left[c_1 \Delta D_1 u + c_2 \Delta D_2 u  \right] = 0\,,
\end{align*}
since $c$ is a constant vector, completing the proof.
\end{proof}

\begin{lemma}
\label{lemma5.5}
Let $T > 0$ and $u_n$ be the solution of \eqref{eq:5.3}. Then under the assumptions $(A.1) - (A.2)$, for all $\rho > 0$ and $p \in [1, 1 + \frac{1}{K_c^2})$, there exist positive constants $C_1(p, \rho)$, $C_2(p, \rho)$ and $C_3(\rho)$ such that if $\|u_0\|_\rV \le \rho$, then
\begin{align}
\label{eq:5.5}
&\sup_{n \ge 1}\mathbb{E}\left(\sup_{r \in [0,T]} \|u_n(r)\|_{\mathrm{V}}^{2p}\right)  \le C_1(p, \rho),\\
\label{eq:5.6}
\sup_{n \ge 1} \E \int_0^T &\|\uns\|_{\mathrm{V}}^{2(p-1)} |\mathrm{A} \uns - |\nabla \uns|_{L^2}^2 \uns|_{\mathrm{H}}^2 ~ds \le C_2(p, \rho),
\end{align}
and
\begin{equation}
\label{eq:5.7}
\sup_{n \ge 1} \E \int_0^T |u_n(s)|_{\mathrm{D}(\mathrm{A})}^2 ~ds \le C_3(\rho).
\end{equation}
\end{lemma}

\begin{proof}
Let $u_n(t)$ be the solution of \eqref{eq:5.3} then applying the It\^o formula to $\phi(x) = \|x\|_{\mathrm{V}}^2$ and the process $u_n(t)$, we get
\begin{align*}
d\|u_n(t)\|_{\mathrm{V}}^2 & = 2 \langle \mathrm{A} u_n(t), - P_n \mathrm{A} u_n(t) - P_n B(u_n(t), u_n(t)) + |\nabla u_n(t)|_{L^2}^2 u_n(t)\rangle_{\mathrm{H}} dt \\
&~~~\sum_{j=1}^m \langle \mathrm{A} \un, (P_n C_j)^2 u_n(t) \rangle_{\mathrm{H}} dt + \sum_{j=1}^m \langle \mathrm{A} P_n C_j u_n(t), P_n C_j u_n(t) \rangle_{\mathrm{H}} dt \\
&~~~ + 2 \sum_{j=1}^m \langle \mathrm{A} u_n(t), P_n C_j u_n(t)\, dW_j(t) \rangle_{\mathrm{H}}.
\end{align*}
Now since $\langle |\nabla \un|_{L^2}^2 \un, \mathrm{A} \un - |\nabla \un|_{L^2}^2 \un \rangle = 0$, using \eqref{eq:2.4}, we have
\begin{align*}
d\|u_n(t)\|^2_{\mathrm{V}} & = -2 \langle \mathrm{A} \un - |\nabla \un|^2_{L^2} \un, \mathrm{A} \un - |\nabla \un|^2 \un \rangle_{\mathrm{H}} dt  \\
&~~~ + 2 \langle |\nabla \un|^2 \un, \mathrm{A} \un - |\nabla \un|_{L^2}^2 \un \rangle_{\mathrm{H}} dt  \\
&~~~ - 2 \langle \mathrm{A} \un, B(\un, \un) \rangle_{\mathrm{H}} dt + \sum_{j=1}^m \langle \mathrm{A} u_n(t), C_j^2 u_n(t) \rangle_{\mathrm{H}} dt\\
&~~~  + \sum_{j=1}^m\langle \mathrm{A} C_j u_n(t), C_j u_n(t) \rangle_{\mathrm{H}} dt  + 2 \sum_{j=1}^m \langle \rm{A} u_n(t), C_j u_n(t)\, dW_j(t) \rangle _{\mathrm{H}} \\
& = - 2 | \mathrm{A} \un - |\nabla \un|_{L^2}^2 \un|_\rH^2 dt +2 \sum_{j=1}^m \langle \mathrm{A} \un, C_j \un\, dW_j(t)\rangle_\rH \\
&~~~ + \sum_{j=1}^m\langle \mathrm{A} C_j \un - C_j \mathrm{A} \un, C_j \un \rangle_\rH\, dt .
\end{align*}
Integrating on both sides and using Assumption \textbf{(A.1)}  and  Lemma~\ref{lemma5.4}, we get
\begin{align}
\label{eq:5.8}
&\| \un \|^2_{\mathrm{V}} + 2 \int_0^t |\mathrm{A} u_n(s) - |\nabla u_n(s)|_{L^2}^2 u_n(s)|_{\mathrm{H}}^2\,ds   \nonumber \\
& = \|u_n(0)\|_{\mathrm{V}}^2  + 2 \sum_{j=1}^m \int_0^t \langle \mathrm{A} u_n(s), C_j u_n(s)\, dW_j(s) \rangle_{\mathrm{H}} \\
& \leq \|u(0)\|_{\mathrm{V}}^2 + 2\sum_{j=1}^m \int_0^t \langle \mathrm{A} u_n(s), C_j u_n(s)\, dW_j(s) \rangle_{\mathrm{H}} \nonumber \,.
\end{align}
By Lemma~\ref{lemma5.2}, we infer that the process
\[\mu_n(t) = \sum_{j=1}^m \int_0^t \langle \mathrm{A} u_n(s), C_j u_n(s)\, dW_j(s) \rangle_{\mathrm{H}}, \quad t \in [0,T]\]
is a martingale and that $\E [ \mu_n(t)] = 0$. Thus
\begin{align}
\label{eq:5.10}
\mathbb{E} \|u_n(t)\|_{\mathrm{V}}^2 + 2 \mathbb{E}\int_0^t | \mathrm{A} u_n(s) - |\nabla u_n(s)|_{L^2}^2 u_n(s) |_{\mathrm{H}}^2~ds  \leq \mathbb{E}\|u(0)\|_{\mathrm{V}}^2.
\end{align}
Hence
\begin{equation}
\label{eq:5.11}
\sup_{n \ge 1} \sup_{t \in [0,T]} \E \|u_n(t)\|^2_{\mathrm{V}} \le \E \|u(0)\|^2_{\mathrm{V}}.
\end{equation}
Note that using (\ref{eq:5.11}) in (\ref{eq:5.10}), we also have the following estimate
\begin{equation}
\label{eq:5.12}
\sup_{n \ge 1}\mathbb{E}\int_0^T |\mathrm{A} u_n(s) - |\nabla u_n(s)|_{L^2}^2 u_n(s)|_{\mathrm{H}}^2~ds \leq  \E \|u(0)\|^2_{\mathrm{V}}.
\end{equation}

Let $\xi(t) = \|u_n(t)\|_{\mathrm{V}}^2$, $t\in [0,T]$ and  $\phi(x) = x^p$, for some fixed $p \in [1, \infty)$. Using the It\^o formula and (\ref{eq:5.8}), we obtain
\begin{align}
\label{eq:5.13}
\|u_n(t)\|_{\mathrm{V}}^{2p} & = \|u_n(0)\|_{\mathrm{V}}^{2p} - 2 p \int_0^t \|\uns\|_{\mathrm{V}}^{2(p-1)} |\mathrm{A} \uns - |\nabla \uns|_{L^2}^2 \uns|_{\mathrm{H}}^2 ds \nonumber \\
&~~~ + 2p(p-1) \sum_{j=1}^m \int_0^t \|\uns\|_{\mathrm{V}}^{2(p-2)} \langle \mathrm{A} \uns, C_j \uns \rangle_{\mathrm{H}}^2~ds \nonumber \\
&~~~ + 2p \sum_{j=1}^m  \int_0^t \|\uns\|_{\mathrm{V}}^{2(p-1)} \langle \mathrm{A} \uns, C_j \uns \,dW_j(s) \rangle_{\mathrm{H}}
\end{align}
Since $C$ is skew symmetric, $\langle C \uns, \uns \rangle = 0$ and hence we get
\begin{align*}
\|u_n(t)\|_{\mathrm{V}}^{2p} & + 2 p \int_0^t \|\uns\|_{\mathrm{V}}^{2(p-1)} |\mathrm{A} \uns - |\nabla \uns|_{L^2}^2 \uns|_{\mathrm{H}}^2 ds \\
& \le \|u_n(0)\|_{\mathrm{V}}^{2p} + 2p \sum_{j=1}^m \int_0^t \|\uns\|_{\mathrm{V}}^{2(p-1)} \langle \mathrm{A} \uns, C_j \uns \,dW_j(s) \rangle_{\mathrm{H}} \\
&~~ + 2p(p-1) \sum_{j=1}^m \int_0^t \|\uns\|_{\mathrm{V}}^{2(p-2)} \langle \mathrm{A} \uns - |\nabla \uns|^2_{L^2} \uns, C_j \uns \rangle_{\mathrm{H}}^2~ds.
\end{align*}
Using the H\"older inequality we have
\begin{align*}
& \|u_n(t)\|_{\mathrm{V}}^{2p} + 2 p \int_0^t \|\uns\|_{\mathrm{V}}^{2(p-1)} |\mathrm{A} \uns - |\nabla \uns|_{L^2}^2 \uns|_{\mathrm{H}}^2 ds\\
&\; \leq \|u_n(0)\|_{\mathrm{V}}^{2p} + 2p \sum_{j=1}^m \int_0^t \|\uns\|_{\mathrm{V}}^{2(p-1)} \langle \mathrm{A} \uns, C_j \uns\,dW_j(s) \rangle_{\mathrm{H}} \\
&\;~~~ + 2p(p-1) \sum_{j=1}^m \int_0^t \|\uns\|_{\mathrm{V}}^{2(p-2)} | \mathrm{A} \uns - |\nabla \uns|_{L^2}^2 \uns|_{\mathrm{H}}^2 |C_j \uns|_{\mathrm{H}}^2\,ds.
\end{align*}
On rearranging we get
\begin{align*}
&\|u_n(t)\|_{\mathrm{V}}^{2p}  + 2 p \int_0^t \|\uns\|_{\mathrm{V}}^{2(p-1)} |\mathrm{A} \uns - |\nabla \uns|_{L^2}^2 \uns|_{\mathrm{H}}^2 ds\\
&\; \leq \|u_n(0)\|_{\mathrm{V}}^{2p} + 2p \sum_{j=1}^m \int_0^t \|\uns\|^{2(p-1)} \langle \mathrm{A} \uns, C_j \uns \,dW_j(s)\rangle_{\mathrm{H}}\\
&\;~~~ + 2p(p-1) K_c^2 \int_0^t \|\uns\|_{\mathrm{V}}^{2(p-1)} | \mathrm{A} \uns - |\nabla \uns|_{L^2}^2 \uns|_{\mathrm{H}}^2\,ds ,
\end{align*}
where $K_c$ is the positive constant defined in equality \eqref{eqn-K_c}.\\ For $p \in [1,1+\frac{1}{K_c^2})$, $K_p = 2p \left[ 1 - K_c^2 (p-1) \right] > 0$, thus
\begin{align}
\label{eq:5.14}
\|u_n(t)\|_{\mathrm{V}}^{2p} & + K_{p} \int_0^t \|\uns\|_{\mathrm{V}}^{2(p-1)} |\mathrm{A} \uns - |\nabla \uns|_{L^2}^2 \uns|_{\mathrm{H}}^2 ds \nonumber \\
&\leq \|u_n(0)\|_{\mathrm{V}}^{2p} + 2p \sum_{j=1}^m \int_0^t \|\uns\|_{\mathrm{V}}^{2(p-1)} \langle \mathrm{A} \uns, C_j \uns\,dW_j(s) \rangle_{\mathrm{H}}.
\end{align}
Using Lemma~\ref{lemma5.2} we infer that the process
\[\eta_n(t) = \sum_{j=1}^m \int_0^t \|\uns\|_{\mathrm{V}}^{2(p-1)} \langle \mathrm{A} \uns, C_j \uns \,dW_j(s) \rangle_{\mathrm{H}}, \quad t \in [0,T],\]
is a martingale and $\E[\eta_n(t)] = 0$. Thus
\begin{align}
\label{eq:5.15}
 \E \|u_n(t)\|_{\mathrm{V}}^{2p} & + K_p \E \int_0^t \|\uns\|_{\mathrm{V}}^{2(p-1)} |\mathrm{A} \uns - |\nabla \uns|^2_{L^2} \uns|_{\mathrm{H}}^2 \,ds & \leq \E \|u_n(0)\|_{\mathrm{V}}^{2p}.
\end{align}
In particular
\begin{equation}
\label{eq:5.17}
\sup_{n \ge 1} \sup_{t \in [0,T]} \E \|u_n(t)\|^{2p} \le \E \|u_0\|_{\mathrm{V}}^{2p}
\end{equation}
Note that using \eqref{eq:5.17} in \eqref{eq:5.15}, we also have the following estimate,
\begin{equation}
\label{eq:5.18}
\sup_{n \ge 1} \E \int_0^T \|\uns\|_{\mathrm{V}}^{2(p-1)} |\mathrm{A} \uns - |\nabla \uns|_{L^2}^2 \uns|_{\mathrm{H}}^2 ~ds \le \frac{1}{K_p}\E \|u_0\|_{\mathrm{V}}^{2p}.
\end{equation}

In order to prove \eqref{eq:5.5} we start from \eqref{eq:5.13},
\begin{align*}
\|u_n(t)\|_{\mathrm{V}}^{2p} & = \|u_n(0)\|_{\mathrm{V}}^{2p} - 2 p \int_0^t \|\uns\|_{\mathrm{V}}^{2(p-1)} |\mathrm{A} \uns - |\nabla \uns|_{L^2}^2 \uns|_{\mathrm{H}}^2 ds \nonumber \\
&~~~ + 2p(p-1) \sum_{j=1}^m \int_0^t \|\uns\|_{\mathrm{V}}^{2(p-2)} \langle \mathrm{A} \uns, C_j \uns \rangle_{\mathrm{H}}^2\,ds \nonumber \\
&~~~ + 2p \sum_{j=1}^m \int_0^t \|\uns\|_{\mathrm{V}}^{2(p-1)} \langle \mathrm{A} \uns, C_j \uns\,dW_j(s) \rangle_{\mathrm{H}}.
\end{align*}
Since for every $j \in \{1, \cdots, m\}$, $\langle C_j \uns, \uns \rangle_{\mathrm{H}} = 0$, hence
\begin{align*}
\|u_n(t)\|_{\mathrm{V}}^{2p} & + 2 p \int_0^t \|\uns\|_{\mathrm{V}}^{2(p-1)} |\mathrm{A} \uns - |\nabla \uns|^2_{L^2} \uns|_{\mathrm{H}}^2\,ds = \|u_n(0)\|_{\mathrm{V}}^{2p}  \nonumber \\
&~~~ + 2p(p-1) \sum_{j=1}^m \int_0^t \|\uns\|_{\mathrm{V}}^{2(p-2)} \langle \mathrm{A} \uns - |\nabla \uns|_{L^2}^2 \uns, C_j \uns \rangle_{\mathrm{H}}^2~ds \nonumber \\
&~~~ + 2p \sum_{j=1}^m \int_0^t \|\uns\|_{\mathrm{V}}^{2(p-1)} \langle \mathrm{A} \uns - |\nabla \uns|^2_{L^2} \uns, C_j \uns\,dW_j(s) \rangle_{\mathrm{H}}.
\end{align*}
Taking the mathematical expectation and using the H\"older inequality, we have
\begin{align}
\label{eq:5.19}
&\mathbb{E} \sup_{r \in [0,t]}  \|u_n(r)\|_{\mathrm{V}}^{2p} + 2p\Eb \sup_{r \in [0,t]} \int_0^r \|\uns\|_{\mathrm{V}}^{2(p-1)} |\mathrm{A} \uns - |\nabla \uns|_{L^2}^2 \uns|_{\mathrm{H}}^2\,ds  \nonumber \\
&\leq \mathbb{E} \|u_n(0)\|_{\mathrm{V}}^{2p} + 2p(p-1)K_c^2 \Eb \sup_{r \in [0,t]} \left[ \int_0^r \|\uns\|^{2(p-2)} | \mathrm{A} \uns - |\nabla \uns|_{L^2}^2 \uns|_{\mathrm{H}}^2  |\nabla \uns |_{L^2}^2\,ds \right] \nonumber \\
&~~ + 2p \Eb \sup_{r \in [0,t]} \left[ \sum_{j=1}^m \int_0^t \|\uns\|_{\mathrm{V}}^{2(p-1)} \langle \mathrm{A} \uns - |\nabla \uns|_{L^2}^2 \uns, C_j \uns\,dW_j(s) \rangle_{\mathrm{H}} \right].
\end{align}
Using the Burkholder- Davis- Gundy inequality, we get
\begin{align*}
\mathbb{E} \sup_{r \in [0,t]} &\left|\sum_{j=1}^m \int_0^r\|\uns\|_{\mathrm{V}}^{2(p-1)} \langle \mathrm{A} \uns - |\nabla \uns |^2_{L^2} \uns, C_j \uns\,dW_j(s) \rangle_{\mathrm{H}} \right| \\
& \leq 3 \Eb \left| \sum_{j=1}^m \int_0^t \|\uns\|_{\mathrm{V}}^{4(p-1)} \langle \mathrm{A} \uns - |\nabla \uns|_{L^2}^2 \uns, C_j \uns \rangle_{\mathrm{H}}^2\,ds \right|^{1/2} \\
& \leq 3 \Eb \left| \sum_{j=1}^m \int_0^t \|\uns\|_{\mathrm{V}}^{4(p-1)} |\mathrm{A} \uns - |\nabla \uns |^2_{L^2} \uns|_{\mathrm{H}}^2 |C_j \uns|_{\mathrm{H}}^2~ds \right|^{1/2}\\
&\leq 3 \Eb K_c \left[ \int_0^t \|\uns\|_{\mathrm{V}}^{2p} \|\uns\|_{\mathrm{V}}^{2(p-1)} |\mathrm{A} \uns - |\nabla \uns |^2_{L^2} \uns|_{\mathrm{H}}^2 ~ds \right]^{1/2}.
\end{align*}
Using the H\"older inequality and the Young inequality, we obtain
\begin{align*}
\mathbb{E} & \sup_{r \in [0,t]} \left|\sum_{j=1}^m \int_0^r\|\uns\|_{\mathrm{V}}^{2(p-1)} \langle \mathrm{A} \uns - |\nabla \uns |_{L^2}^2 \uns, C_j \uns\,dW_j(s) \rangle_{\mathrm{H}} \right| \\
& \leq 3 \Eb \left[ K_c \left(\sup_{r \in [0,t]} \|u_n(r)\|_{\mathrm{V}}^{2p} \right)^{1/2} \left(\int_0^t \|\uns\|_{\mathrm{V}}^{2(p-1)}|\mathrm{A} \uns - |\nabla \uns |_{L^2}^2 \uns |_{\mathrm{H}}^2~ds \right)^{1/2} \right]\\
& \leq 3 \Eb \left[ \varepsilon \sup_{r \in [0,t]} \|u_n(r)\|_{\mathrm{V}}^{2p} + \dfrac{K_c^2}{4 \varepsilon} \int_0^t \|\uns\|_{\mathrm{V}}^{2(p-1)} |\mathrm{A} \uns - |\nabla \uns |_{L^2}^2 \uns |_{\mathrm{H}}^2 ~ds \right].
\end{align*}
Thus using this in \eqref{eq:5.19}, we get
\begin{align}
\label{eq:5.20}
\Eb \sup_{r \in [0,t]} & \|u_n(r)\|_{\mathrm{V}}^{2p}  + 2 \mathbb{E} \sup_{r \in [0,t]} \int_0^r \|\uns\|_{\mathrm{V}}^{2(p-1)} |\mathrm{A} \uns - |\nabla \uns|_{L^2}^2 \uns|_{\mathrm{H}}^2~ds \nonumber \\
& \leq \mathbb{E} \|u_n(0)\|_{\mathrm{V}}^{2p} + 2p(p-1) K_c^2 \Eb \sup_{r \in [0,t]} \int_0^r \|\uns\|_{\mathrm{V}}^{2(p-1)} | \mathrm{A} \uns - |\nabla \uns |_{L^2}^2 \uns|_{\mathrm{H}}^2 ~ds \nonumber \\
&~~~  + \dfrac{3 p K_c^2}{2 \varepsilon} \Eb\int_0^t \|\uns\|_{\mathrm{V}}^{2(p-1)}|\mathrm{A} \uns - |\nabla \uns |^2_{L^2} \uns |_{\mathrm{H}}^2 \,ds
\end{align}
Hence for $\varepsilon = \frac{1}{12p}$, Eq. \eqref{eq:5.20} reduces to
\begin{align*}
\Eb \sup_{r \in [0,t]}& \|u_n(r)\|_{\mathrm{V}}^{2p}  + 4 \mathbb{E} \sup_{r \in [0,t]} \int_0^r \|\uns\|_{\mathrm{V}}^{2(p-1)} |\mathrm{A} \uns - |\nabla \uns|_{L^2}^2 \uns|_{\mathrm{H}}^2~ds \leq 2 \mathbb{E} \|u_n(0)\|_{\mathrm{V}}^{2p}\\
&~~~ + 4p(p-1)K_c^2 \Eb \sup_{r \in [0,t]} \int_0^r \|\uns\|_{\mathrm{V}}^{2(p-1)} | \mathrm{A} \uns - |\nabla \uns |_{L^2}^2 \uns|_{\mathrm{H}}^2 ~ds\\
&~~~ + 36p^2 K_c^2 \Eb \int_0^t \|\uns\|_{\mathrm{V}}^{2(p-1)} | \mathrm{A} \uns - |\nabla \uns|_{L^2}^2 \uns|_{\mathrm{H}}^2 ~ds.
\end{align*}
Since $\int_0^r |\mathrm{A} \uns - |\nabla \uns |_{L^2}^2 \uns |_\rH^2 ~ds$ is an increasing function, we have
\begin{align*}
\Eb & \sup_{r \in [0,t]}  \|u_n(r)\|_{\mathrm{V}}^{2p}  + 4 \mathbb{E} \int_0^t \|\uns\|_{\mathrm{V}}^{2(p-1)} |\mathrm{A} \uns - |\nabla \uns|_{L^2}^2 \uns|_{\mathrm{H}}^2~ds\\
& \leq 2 \mathbb{E} \|u_n(0)\|_{\mathrm{V}}^{2p} + 4pK_c^2 \left[10p -1\right]\Eb \int_0^t \|\uns\|_{\mathrm{V}}^{2(p-1)} | \mathrm{A} \uns - |\nabla \uns|^2_{L^2} \uns|_{\mathrm{H}}^2 ~ds.
\end{align*}
In particular
\begin{align*}
\Eb  \sup_{r \in [0,t]}  \|u_n(r)\|_{\mathrm{V}}^{2p}  \le & 4 p K_c^2 \left[10p -1\right] \Eb\int_0^t \|\uns\|_{\mathrm{V}}^{2(p-1)} |\mathrm{A} \uns - |\nabla \uns |_{L^2}^2 \uns |_{\mathrm{H}}^2 ~ds  \\
&~~~+ 2 \mathbb{E} \|u_n(0)\|_{\mathrm{V}}^{2p}\,.
\end{align*}
Since $\Eb \|u_n(0)\|_{\mathrm{V}}^{2p} \le \E \|u_0\|^{2p}_{\mathrm{V}}$ and using \eqref{eq:5.18}, for $p \in [1, 1 + \frac{1}{K_c^2})$
\[\Eb \int_0^T \|\uns\|_{\mathrm{V}}^{2(p-1)} |\mathrm{A} \uns - |\nabla \uns |_{L^2}^2 \uns|_{\mathrm{H}}^2\,ds\]
is uniformly bounded in $n$, thus
\[\sup_{n \ge 1}\Eb \sup_{r \in [0,T]} \|u_n(r)\|_{\mathrm{V}}^{2p} \leq C_1(p, \rho).\]

Now we will establish \eqref{eq:5.7}. Note that
\[
\Eb\int_0^T |\uns|^2_{\mathrm{D}(\mathrm{A})} ds = \Eb\int_0^T |\mathrm{A} \uns - |\nabla \uns |_{L^2}^2 \uns |_{\mathrm{H}}^2\,ds + \E \int_0^T \|\uns\|^4\, ds. \]
Using \eqref{eq:5.5} for $p=2$ and \eqref{eq:5.6} for $p = 1$, we get
\[ \sup_{n \ge 1} \Eb \int_0^T |\uns|^2_{\mathrm{D}(\mathrm{A})} ds \leq C_2(1,\rho) + C_1(2, \rho)T =: C_3(\rho). \]
\end{proof}

\subsection{Tightness}
\label{s:5.2}
In this subsection using the a'priori estimates from the Lemma~\ref{lemma5.5} and the Corollary~\ref{cor3.6} we will prove that for every $n \in \N$ the measures $\mathcal{L}(u_n)$ on $(\mathcal{Z}_T, \mathcal{T})$ defined by the solution of the stochastic ODE \eqref{eq:5.3}
are tight. The following is the main result of this subsection.

\begin{lemma}
\label{lemma5.6}
The set of measures $\left\{\mathcal{L}(u_n), n \in \mathbb{N}\right\}$ is tight on $(\mathcal{Z}_T, \mathcal{T})$.
\end{lemma}

\begin{proof}
We apply Corollary~\ref{cor3.6}. According to the a'priori estimates \eqref{eq:5.5} (for $p =1$) and \eqref{eq:5.7}, conditions $(a)$ and $(b)$ of Corollary~\ref{cor3.6} are satisfied. Thus it is sufficient to prove that the sequence $(u_n)_{n \in \mathbb{N}}$ satisfies the Aldous condition $[\mathbf{A}]$ in $\rH$. Let $(\tau_n)_{n \in \mathbb{N}}$ be a sequence of stopping times such that $0 \leq \tau_n \leq T$. By \eqref{eq:5.3}, for $t \in [0,T]$ we have
\begin{align*}
\un &= u_n(0) - \int_0^t P_n \mathrm{A} \uns\,ds - \int_0^t P_n B (\uns)\,ds + \int_0^t |\nabla \uns|_{L^2}^2 \uns\,ds \\
&~~~ + \dfrac{1}{2} \int_0^t (P_nC)^2 \uns\,ds + \int_0^t P_n C \uns\,dW(s)\\
&:= J_1^n + J_2^n(t) + J_3^n(t) +J_4^n(t) + J_5^n(t) + J_6^n(t),\quad t \in [0,T].
\end{align*}
Let $\theta > 0$. First we make some estimates for each term of the above equality.\\

\noindent \textbf{Ad.}$J_2^n$. Since $\mathrm{A} : \mathrm{D}(\mathrm{A}) \rightarrow \rH$, then by the H\"older inequality and \eqref{eq:5.7}, we have the following estimates
\begin{align}
\label{eq:5.21}
\Eb &\left[ |J_2^n(\tau_n + \theta) - J_2^n(\tau_n)|_{\mathrm{H}} \right] = \Eb \left| \int_{\tau_n}^{\tau_n + \theta} P_n \mathrm{A} \uns~ds \right|_{\mathrm{H}} \leq c\Eb \int_{\tau_n}^{\tau_n + \theta} |\mathrm{A}\uns|_{\mathrm{H}}~ds \nonumber \\
& \leq c\Eb \int_{\tau_n}^{\tau_n + \theta} |\uns|_{\mathrm{D}(\mathrm{A})}~ds \leq c\theta^{\frac12} \left(\E\left[ \int_0^T |\uns|^2_{\rD(\mathrm{A})}\,ds \right]\right)^{\frac12} \leq c C_3^\frac12 \cdot \theta^{\frac12} =: c_2 \cdot \theta^{\frac12}.
\end{align}

\noindent  \textbf{Ad.} $J^n_3$. Since $B: \rV \times \rV \rightarrow \rH$ is bilinear and continuous, then using \eqref{eq:2.3}, the Cauchy-Schwarz inequality, \eqref{eq:5.5} and \eqref{eq:5.7}, we have the following estimates
\begin{align}
\label{eq:5.22}
\Eb  [ | J^n_3(\tau_n &+ \theta)  - J^n_3(\tau_n)|_{\mathrm{H}}] = \Eb \left|\int_{\tau_n}^{\tau_n + \theta} P_n B(\uns)\,ds \right|_{\mathrm{H}}  \leq c \Eb \int_{\tau_n}^{\tau_n + \theta} |B(\uns, \uns)|_{\mathrm{H}}\,ds \nonumber \\
& \leq c \Eb \int_{\tau_n}^{\tau_n + \theta} |\uns|^{\frac12}_{\mathrm{H}} \|\uns\|_{\mathrm{V}} |\uns|_{\mathrm{D}(\mathrm{A})}^{\frac12}~ds \leq c\Eb \left[ \int_{\tau_n}^{\tau_n + \theta} \|\uns\|^{\frac32}_{\mathrm{V}} |\uns|_{\mathrm{D}(\mathrm{A})}^{1/2}~ds\right]\nonumber \\
& \leq c\Eb \left(\left[\int_{\tau_n}^{\tau_n + \theta} \|\uns\|_\rV^2\,ds \right]^{\frac34} \left[ \int_{\tau_n}^{\tau_n + \theta} |\uns|^2_{\mathrm{D}(\mathrm{A})}\,ds \right]^{\frac14}\right)\nonumber \\
& \leq c \theta^{\frac34} \left[\E \sup_{s \in [0,T]}\|\uns\|^2_{\mathrm{V}} \right]^{\frac34} \left[ \Eb \int_{0}^{T} |\uns|^2_{\mathrm{D}(\mathrm{A})}\,ds \right]^{\frac14} \leq c C_1(1)^{\frac34} C_3^{\frac14}\cdot \theta^{\frac34} =: c_3 \cdot \theta^{\frac34}.
\end{align}

\noindent \textbf{Ad.} $J^n_4$. Using Lemma~\ref{lemma5.3} and estimate \eqref{eq:5.5}, we have
\begin{align}
\label{eq:5.23}
\Eb  [|J^n_4(\tau_n + \theta) & - J^n_4(\tau_n)|_{\mathrm{H}}] = \Eb \left| \int_{\tau_n}^{\tau_n + \theta} |\nabla \uns|_{L^2}^2 \uns~ds \right|_{\mathrm{H}} \nonumber \\
& \leq \Eb \int_{\tau_n}^{\tau_n + \theta} |\nabla \uns|_{L^2}^2 |\uns|_{\mathrm{H}}~ds \leq \Eb \sup_{s \in [0,T]} \|\uns\|_\rV^2 \,\theta \leq C_1(1) \cdot \theta =: c_4 \cdot \theta.
\end{align}

\noindent \textbf{Ad.} $J^n_5$. Since $C$ is linear and continuous, then using the Cauchy-Schwarz inequality, Assumption $(A.1)$ and \eqref{eq:5.7}, we have the following
\begin{align}
\label{eq:5.24}
\Eb [ | J^n_5(\tau_n + \theta) & - J^n_5(\tau_n)|_{\mathrm{H}} ]  = \Eb \left| \dfrac{1}{2} \sum_{j=1}^m \int_{\tau_n}^{\tau_n + \theta} (P_n C_j)^2 \uns \,ds \right|_{\mathrm{H}}  \nonumber \\
& \leq \dfrac{1}{2} c \Eb \left( \sum_{j=1}^m \int_{\tau_n}^{\tau_n + \theta} |C_j^2 \uns|_{\mathrm{H}}\,ds \right) \leq \dfrac{1}{2} c K_c^2  \Eb\int_{\tau_n}^{\tau_n + \theta} |\uns|_{\mathrm{D}(\mathrm{A})}~ds  \nonumber \\
& \leq \dfrac{1}{2} c K_c^2 \left[ \Eb \int_{0}^{T} |\uns|^2_{\mathrm{D}(\mathrm{A})}\,ds \right]^{\frac12} \theta^{\frac12} \leq \frac{cK_c^2}{2} C_3^\frac12 \cdot \theta^{\frac12} =: c_5 \cdot \theta^\frac12.
\end{align}

\noindent \textbf{Ad.} $J^n_6$. Using the Ito isometry, Assumption $(A.1)$ and estimate \eqref{eq:5.5}, we obtain the following
\begin{align}
\label{eq:5.25}
\Eb & \left[ | J^n_6(\tau_n + \theta) - J^n_6(\tau_n)|^2_{\mathrm{H}} \right] = \Eb \left| \int_{\tau_n}^{\tau_n + \theta} P_n C \uns\,dW(s)\right|^2_{\mathrm{H}} \leq c \Eb \int_{\tau_n}^{\tau_n + \theta} |C \uns|_\rH^2\,ds \nonumber \\
& \leq c K_c^2 \Eb \int_{\tau_n}^{\tau_n + \theta} \|\uns\|_\rV^2\,ds  \leq c K_c^2 \Eb \sup_{s \in [0,T]}\|\uns\|_\rV^2\, \theta  \leq c K_c C_1(1) \cdot \theta =: c_6 \cdot \theta.
\end{align}
Let us fix $\kappa > 0$ and $\varepsilon > 0$. By the Chebyshev's inequality and estimates \eqref{eq:5.21} - \eqref{eq:5.24}, we obtain
\begin{align*}
\mathbb{P}( \{|J_i^n(\tau_n + \theta) - J^n_i(\tau_n)|_{\mathrm{H}} \geq \kappa \}) & \leq \dfrac{1}{\kappa} \mathbb{E} \left[|J_i^n(\tau_n+ \theta) - J^n_i(\tau_n)|_{\mathrm{H}} \right] \leq \dfrac{c_i \theta}{\kappa}; \quad n \in \mathbb{N},
\end{align*}
where $i = 1, \cdots, 5$. Let $\delta_i = \dfrac{\kappa}{c_i} \varepsilon$. Then
\[ \sup_{n \in \mathbb{N}} \sup_{0 \leq \theta \leq \delta_i} \mathbb{P}( \{|J_i^n(\tau_n + \theta) - J^n_i(\tau_n)|_{\mathrm{H}} \geq \kappa \}) \leq \varepsilon,~~~ i = 1 \dots 5.\]
By the Chebyshev inequality and \eqref{eq:5.25}, we have
\begin{align*}
\mathbb{P}( \{|J_6^n(\tau_n + \theta) - J^n_6(\tau_n)|_{\mathrm{H}} \geq \kappa \}) &\leq \dfrac{1}{\kappa^2} \mathbb{E} \left[|J_i^n(\tau_n+ \theta) - J^n_i(\tau_n)|_{\mathrm{H}}^2 \right]\\
& \leq \dfrac{c_6 \theta}{\kappa^2},~~~n \in \mathbb{N}.
\end{align*}
Let $\delta_6 = \dfrac{\kappa^2}{C_6} \varepsilon$. Then
\[ \sup_{n \in \mathbb{N}} \sup_{0 \leq \theta \leq \delta_6} \mathbb{P}( \{|J_6^n(\tau_n + \theta) - J^n_6(\tau_n)|_{\mathrm{H}} \geq \kappa \}) \leq \varepsilon.\]
Since $[\textbf{A}]$ holds for each term $J_i^n,~i= 1,2, \cdots, 6$; we infer that it holds also for $(u_n)$. Therefore, we can conclude the proof of the lemma by invoking Corollary~\ref{cor3.6}.
\end{proof}

\subsection{Proof of Theorem~\ref{thm5.1}}
\label{s:5.3}

By Lemma~\ref{lemma5.6} the set of measures $\{\mathcal{L}(u_n), n \in \mathbb{N} \}$ is tight on the space $(\mathcal{Z}_T, \mathcal{T})$ defined by \eqref{eq:3.1}. Hence by Corollary~\ref{cor3.8} there exist a subsequence $(n_k)_{k \in \mathbb{N}}$, a probability space $(\tilde{\Omega}, \tilde{\mathcal{F}}, \tilde{\mathbb{P}})$ and, on this space, $\mathcal{Z}_T$-valued random variables $\tilde{u}, \tilde{u}_{n_k}, k \ge 1$ such that
\begin{equation}
\label{eq:5.26}
\tilde{u}_{n_k}\,\mbox{has the same law as }\, u_{n_k}\,\mbox{ and }\, \tilde{u}_{n_k} \to \tilde{u} \,\mbox{ in }\,\mathcal{Z}_T,\quad \tp-\mbox{a.s.}
\end{equation}
$\tunk \to \tu$ in $\mathcal{Z}_T$, $\tp-\mbox{a.s.}$ precisely means that
\begin{align*}
\tu_{n_k} &\to \tu \, \mbox{in }\ccal([0,T];\rH),\\
\tu_{n_k} &\rightharpoonup \tu \, \mbox{in }L^2(0,T; \rD(\mathrm{A})),\\
\tu_{n_k} &\to \tu \, \mbox{in }L^2(0,T; \rV),\\
\tu_{n_k} &\to \tu \, \mbox{in }\ccal([0,T]; \rV_{\mathrm{w}}).
\end{align*}

Let us denote the subsequence $(\tilde{u}_{n_k})$ again by $(\tun)_{n \in \mathbb{N}}$.\\

Since $u_n \in \ccal([0,T]; \mathrm{H}_n), \mathbb{P}$-a.s. and $\ccal([0,T]; \rH_n)$ is a Borel subset of $\ccal([0,T];\mathrm{H}) \cap L^2(0,T; \rV)$ and also $\tun$, $u_n$ have the same laws on $\mathcal{Z}_T$ we can make the following inferences
\begin{align*}
&\mathcal{L}(\tun)(\ccal([0,T]; \mathrm{H}_n) = 1, \quad n \ge 1\;,\\
&|\tilde{u}_n(t)|_\rH = |u_n(t)|_\rH,\, a.s.
\end{align*}
Also from \eqref{eq:5.26} $\tilde{u}_n \to \tilde{u}$ in $\ccal([0,T]; \rH)$ and by Lemma~\ref{lemma5.3} $u_n(t) \in \mathcal{M}$ for every $t \in [0,T]$. Therefore we can conclude that
\begin{equation}
\label{eq:invariance}
\tilde{u}(t) \in \mathcal{M}, \quad t \in [0,T].
\end{equation}
Moreover, by \eqref{eq:5.5} and \eqref{eq:5.7}, for $p \in [1, 1 + \frac{1}{K_c^2})$
\begin{align}
\label{eq:5.27}
\sup_{n \in \mathbb{N}} \tilde{\E} \left(\sup_{0 \le s \le T} \|\tuns\|^{2p}_{\mathrm{V}} \right) \le C_1(p),\\
\label{eq:5.28}
\sup_{n \in \mathbb{N}} \tilde{\E} \left[ \int_0^T |\tuns|^2_{\mathrm{D}(\mathrm{A})}\,ds \right] \le C_3.
\end{align}

By inequality \eqref{eq:5.28} we infer that the sequence $(\tun)$ contains a subsequence, still denoted by $(\tun)$ convergent weakly in the space $L^2([0,T] \times \tom; \mathrm{D}(\mathrm{A}))$. Since by \eqref{eq:5.26} $\tp$-a.s $\tun \to \tilde{u}$ in $\mathcal{Z}_T$, we conclude that $\tilde{u} \in L^2([0,T] \times \tom; \mathrm{D}(\mathrm{A}))$, i.e.
\begin{equation}
\label{eq:5.29}
\tilde{\E} \left[\int_0^T |\tus|^2_{\mathrm{D}(\mathrm{A})}\,ds \right] < \infty.
\end{equation}

Similarly by inequality \eqref{eq:5.27} we can choose a subsequence of $(\tun)$ convergent weak star in the space $L^2(\tom; L^\infty(0,T; \mathrm{V}))$ and, using \eqref{eq:5.26}, we infer that
\begin{equation}
\label{eq:5.30}
\tilde{\E} \left(\sup_{0 \le s \le T} \|\tus\|^{2}_{\mathrm{V}} \right) < \infty.
\end{equation}

For each $n \geq 1$, let us consider a process $\tilde{M}_n$ with trajectories in $\ccal([0,T]; \rH_n)$, in particular in $\ccal([0,T];\mathrm{H})$ defined by
\begin{align}
\label{eq:5.31}
\tilde{M}_n(t) & = \tilde{u}_n(t) - P_n \tilde{u}(0) + \int_0^t P_n \mathrm{A} \tuns\,ds + \int_0^t P_n B(\tuns)\,ds \nonumber \\
&~~~ - \int_0^t |\nabla \tuns|^2 \tuns \, ds - \dfrac12 \sum_{j=1}^m \int_0^t (P_nC_j)^2 \tuns \, ds \quad t \in [0,T].
\end{align}

\begin{lemma}
\label{lemma5.7}
$\tilde{M}_n$ is a square integrable martingale with respect to the filtration $\tilde{\mathbb{F}}_n = (\tilde{\mathcal{F}}_{n,t})$, where $\tilde{\mathcal{F}}_{n,t} = \sigma\{\tuns, s \leq t\}$ with the quadratic variation
\begin{equation}
\label{eq:5.32}
 \langle \langle \tilde{M}_n \rangle \rangle_t =  \int_0^t \sum_{j=1}^m |P_n C_j \tuns|^2_\rH \, ds.
\end{equation}
\end{lemma}

\begin{proof}
Indeed, since $\tun$ and $u_n$ have the same laws, for all $s, t \in [0,T], s \leq t$, for all bounded continuous functions $h$ on $\ccal([0,s]; \mathrm{H})$, and all $\psi, \zeta \in \mathrm{H}$, we have

\begin{equation}
\label{eq:5.33}
\tilde{\Eb} \left[ \langle \tilde{M}_n(t) - \tilde{M}_n(s) , \psi \rangle_{\mathrm{H}} h(\tilde{u}_{n|[0,s]}) \right] = 0
\end{equation}
and
\begin{align}
\label{eq:5.34}
\tilde{\Eb} \Big[ &\Big( \langle \tilde{M}_n(t), \psi \rangle_{\mathrm{H}} \langle \tilde{M}_n(t), \zeta \rangle_{\mathrm{H}} - \langle \tilde{M}_n(s), \psi \rangle_{\mathrm{H}} \langle \tilde{M}_n(s), \zeta \rangle_{\mathrm{H}} \nonumber \\
& - \sum_{j=1}^m \int_s^t \left\langle \left(C_j \tunsi\right)^\ast P_n \psi , \left(C_j \tunsi\right)^\ast P_n \zeta \right\rangle_{\R}\, d \sigma \Big) \cdot h(\tilde{u}_{n|[0,s]}) \Big] = 0.
\end{align}
\end{proof}

\begin{lemma}
\label{lemma_mart}
Let us define a process $\tilde{M}$ for $t \in [0,T]$ by
\begin{align}\nonumber
\label{eq:5.35}
\tilde{M}(t) & = \tilde{u}(t) - \tilde{u}(0) + \int_0^t \mathrm{A} \tus\,ds + \int_0^t B(\tus)\,ds
\\ &- \int_0^t |\nabla \tus |_{L^2}^2 \tus\,ds - \dfrac12  \sum_{j=1}^m \int_0^t C_j^2 \tus\,ds.
\end{align}
Then $\tilde{M}$ is an $\rH-$valued continuous process.
\end{lemma}

\begin{proof}
Since $\tilde{u} \in \ccal([0,T]; \rV)$ we just need to show that each of the remaining four terms on the RHS of \eqref{eq:5.35} are $\rH-$valued and well defined.\\

Using the Cauchy-Schwarz inequality repeatedly and by \eqref{eq:5.29} we have the following inequalities
\begin{align*}
\tilde{\E}\,\int_0^T |\mathrm{A}\tus|_\rH\,ds \le T^{1/2} \left({\tilde{\E}} \int_0^T |\tus|^2_{\mathrm{D}(\mathrm{A})}\,ds\right)^{1/2} < \infty.
\end{align*}

Using \eqref{eq:2.3}, the H\"older inequality, \eqref{eq:invariance} and the estimates \eqref{eq:5.29} and \eqref{eq:5.30} we obtain the following:
\begin{align*}
\tilde{\E} & \int_0^T|B(\tus)|_\rH\,ds \le 2 \tilde{\E} \int_0^T |\tus|^{1/2}_\rH |\nabla \tus|_{L^2}|\tus|^{1/2}_{\mathrm{D}(\mathrm{A})}\,ds \\
& \le 2 \tilde{\E}\left[\left(\int_0^T \|\tus\|_\rV^{4/3}\,ds\right)^{3/4} \left(\int_0^T |\tus|^2_{\mathrm{D}(\mathrm{A})}\,ds \right)^{1/4}\right]\\
& \le 2 T^{3/4} \left(\tilde{\E}\sup_{s \in [0,T]}\|\tus\|^{4/3}_\rV\right)^{3/4}\left(\tilde{\E}\int_0^T |\tus|^2_{\mathrm{D}(\mathrm{A})}\,ds \right)^{1/4}< \infty.
\end{align*}
Using the H\"older inequality, \eqref{eq:invariance} and inequality \eqref{eq:5.30} we have
\begin{align*}
\tilde{\E}& \int_0^T |\nabla \tus|^2_{L^2}|\tus|_\rH \,ds \le \tilde{\E}\int_0^T \|\tus\|^2_\rV \,ds \le \tilde{\E} \left( \sup_{s \in [0,T]}\|\tus\|^2_\rV \right) T < \infty.
\end{align*}

Now we are left to deal with the last term on the RHS. Using Assumption $(A.1)$ and  estimate \eqref{eq:5.29}, we have the following inequalities for every $j \in \{1, \cdots, m \}$,
\begin{align*}
\tilde{\E} \int_0^T|C_j^2 \tus|_\rH \le K_c T^{1/2} \left(\tilde{\E} \int_0^T |\tus|^2_{\mathrm{D}(\mathrm{A})} \,ds \right)^{1/2} < \infty.
\end{align*}
This concludes the proof of the lemma.
\end{proof}

\begin{lemma}
\label{lemma5.8}
For all $s, t \in [0,T]$ such that $s \leq t$ then:
\begin{itemize}
\item[(a)] $\lim_{n \rightarrow \infty}\langle \tilde{u}_n(t), P_n \psi\rangle_{\mathrm{H}} = \langle\tilde{u}(t), \psi \rangle_{\mathrm{H}},~~ \tilde{\mathbb{P}}$-a.s.~~~ $\psi \in \mathrm{H}$,
\item[(b)] $\lim_{n \rightarrow \infty} \int_s^t \langle \mathrm{A} \tunsi, P_n \psi \rangle_{\mathrm{H}}~d\sigma = \int_s^t \langle \mathrm{A} \tusi, \psi \rangle_{\mathrm{H}}~d\sigma,~~ \tilde{\mathbb{P}}$-a.s.~~ $\psi \in \mathrm{H}$,
\item[(c)]
 $\lim_{n \rightarrow \infty} \int_s^t \langle B(\tunsi, \tunsi), P_n \psi \rangle_{\mathrm{H}}~d\sigma = \int_s^t\langle B(\tusi, \tusi), \psi \rangle\,d \sigma,~~ \tilde{\mathbb{P}}$-a.s.~~ $\psi \in \rV$,
\item[(d)] $\lim_{n \rightarrow \infty} \int_s^t |\nabla \tunsi |_{L^2}^2 \langle \tunsi, P_n \psi \rangle_{\mathrm{H}}~d \sigma = \int_s^t |\nabla \tusi|_{L^2}^2 \langle \tusi, \psi \rangle_{\mathrm{H}}~d \sigma,~ \tilde{\mathbb{P}}$-a.s. $\psi \in \mathrm{H}$,
\item[(e)] $\lim_{n \rightarrow \infty} \langle \int_s^t C_j^2 \tunsi, P_n \psi \rangle_{\mathrm{H}}~d \sigma = \int_s^t \langle C_j^2 \tusi, \psi \rangle_{\mathrm{H}}~d \sigma,~~ \tilde{\mathbb{P}}$-a.s.~~ $\psi \in \mathrm{H}$.
\end{itemize}
\end{lemma}

\begin{proof}
Let us fix $s, t \in [0,T]$, $s \leq t$. By \eqref{eq:5.26} we know that
\begin{equation}
\label{eq:5.36}
\tilde{u}_n \rightarrow \tilde{u}~in~\ccal([0,T];\mathrm{H}) \cap L^2_{\mathrm{w}}(0,T; D(\mathrm{A})) \cap L^2(0,T; V) \cap \ccal([0,T]; V_{\mathrm{w}}),~~~\tilde{\mathbb{P}}\text{-a.s.}
\end{equation}

Let $\psi \in \rH$. Since $\tilde{u}_n \rightarrow \tilde{u}$ in $\ccal([0,T]; \mathrm{H})$ $\tilde{\mathbb{P}}$-a.s. and $P_n \psi \rightarrow \psi$ in $\mathrm{H}$, we have
\begin{align*}
\lim_{n \to \infty} & \langle\tilde{u}_n(t), P_n \psi\rangle_{\mathrm{H}} - \langle\tilde{u}(t), \psi\rangle_{\mathrm{H}} \\
& = \lim_{n \to \infty}\langle\tilde{u}_n(t) - \tilde{u}(t), P_n \psi\rangle_{\mathrm{H}} + \lim_{n \to \infty}\langle\tilde{u}(t), P_n \psi - \psi\rangle_{\mathrm{H}} = 0 \quad \tilde{\mathbb{P}}\text{-a.s.}
\end{align*}
Thus we infer that assertion $(a)$ holds.\\

Let $\psi \in \rH$, then
\begin{align*}
&\int_s^t \langle \mathrm{A} \tunsi, P_n \psi \rangle_{\mathrm{H}} ~ d \sigma - \int_s^t \langle \mathrm{A} \tusi, \psi\rangle_{\mathrm{H}}~d \sigma \\
& = \int_s^t \langle \mathrm{A} \tunsi - \mathrm{A} \tusi, \psi\rangle_{\mathrm{H}}~d \sigma + \int_s^t \langle \mathrm{A} \tunsi, P_n \psi - \psi\rangle_{\mathrm{H}}~d \sigma\\
& \leq \int_s^t \langle \tunsi - \tusi, \mathrm{A}^{-1} \psi\rangle_{\mathrm{D}(\mathrm{A})}\,d \sigma + \int_s^t |\tunsi|_{\mathrm{D}(\mathrm{A})} |P_n \psi - \psi|_{\mathrm{H}}~d \sigma\\
& \leq \int_s^t \langle \tunsi - \tusi, \mathrm{A}^{-1} \psi\rangle_{\mathrm{D}(\mathrm{A})}\,d \sigma + |P_n \psi - \psi|_{\mathrm{H}} |\tilde{u}_n|_{L^2(0,T;\mathrm{D}(\mathrm{A}))} T^{1/2}.
\end{align*}
By \eqref{eq:5.36} $\tilde{u}_n \rightarrow \tilde{u}$ weakly in $L^2(0,T; \mathrm{D}(\mathrm{A}))$ $\tilde{\mathbb{P}}$-a.s. $\tilde{u}_n$ is a uniformly bounded sequence in $L^2(0,T; \mathrm{D}(\mathrm{A}))$ and $P_n \psi \rightarrow \psi$ in $\mathrm{H}$. Hence we have, $\tilde{\mathbb{P}}-$a.s.,
\[\lim_{n \rightarrow \infty} \int_s^t \langle \tunsi - \tusi, \mathrm{A}^{-1} \psi \rangle_{\mathrm{D}(\mathrm{A})} ~ d \sigma \to 0,\]
and \[
\lim_{n \rightarrow \infty}|P_n \psi - \psi|_{\mathrm{H}} \to 0.\]
Thus, we have shown that assertion $(b)$ is true.\\

We will now prove assertion $(c)$. Let $\psi \in \rV$. Then we have the following estimates:
\begin{align*}
&\int_s^t \langle B(\tunsi), P_n \psi \rangle_\rH - \int_s^t \langle B(\tusi), \psi\rangle \,d \sigma \\
&~~ = \int_s^t\langle B(\tunsi) - B(\tusi), \psi \rangle_\rH \,d \sigma + \int_s^t \langle B(\tunsi), P_n \psi - \psi\rangle \, d \sigma \\
&~~ = \int_s^t \left[b(\tunsi, \tunsi, \psi) - b(\tusi, \tusi, \psi) \right]\,d \sigma + \int_s^t \langle B(\tunsi), P_n \psi - \psi\rangle \, d \sigma.
\end{align*}
Using \eqref{eq:2.2}, we get
\begin{align*}
&\int_s^t \langle B(\tunsi), P_n \psi \rangle_\rH - \int_s^t \langle B(\tusi), \psi\rangle \,d \sigma \\
&~~ =  \int_s^t b(\tunsi - \tusi, \tunsi, \psi) \,d \sigma + \int_s^t b(\tusi, \tunsi - \tusi, \psi)~d \sigma \\
&~~~~ + \int_s^t \langle B(\tunsi), P_n \psi - \psi\rangle \, d \sigma \\
&~~ \leq  \int_s^t \|\tunsi - \tusi\|_{\mathrm{V}} \|\tunsi\|_{\mathrm{V}} \|\psi\|_{\mathrm{V}}\,d \sigma  +  \int_s^t \|\tusi\|_{\mathrm{V}}\|\tunsi -\tusi\|_{\mathrm{V}} \|\psi\|_{\mathrm{V}}~d \sigma\\
&~~~~ +  \int_s^t \|\tunsi\|^2_{\mathrm{V}}\|P_n \psi - \psi\|_{\mathrm{V}}\,d \sigma .
\end{align*}
Now since, $\tilde{u}_n \rightarrow \tilde{u}$ in $L^2(0,T; \rV)$, in particular $\tilde{u} \in L^{2}(0,T; \rV)$, also the sequence $(\tilde{u}_n)$ is uniformly bounded in $L^2(0,T; \rV)$. Thus using the Cauchy-Schwarz inequality and the convergence of $P_n \psi \to \psi$ in $\rV$, we have $\tilde{\mathbb{P}}-$a.s.,
\begin{align*}
&\lim_{n \to \infty}\int_s^t \langle B(\tunsi), P_n \psi \rangle_\rH - \int_s^t \langle B(\tusi), \psi\rangle \,d \sigma \\
&~~ \leq \lim_{n \to \infty} |\tilde{u}_n - \tilde{u}|_{L^2(0,T; \mathrm{V})} \Bigl[ |\tilde{u}_n|_{L^2(0,T; \mathrm{V})}
\\ &+ |\tilde{u}|_{L^2(0,T; \mathrm{V})} \Bigr] \|\psi\|_{\mathrm{V}} + \lim_{n \to \infty} |\tilde{u}_n|^2_{L^2(0,T;V)}\|P_n \psi - \psi\|_\rV \rightarrow 0.
\end{align*}

Next we deal with $(d)$. Let $\psi \in \rH$, then
\begin{align*}
&\int_s^t |\nabla \tunsi|_{L^2}^2 \langle \tunsi, P_n \psi\rangle_{\mathrm{H}}\,d \sigma - \int_s^t |\nabla \tusi |_{L^2}^2\langle \tusi, \psi\rangle_{\mathrm{H}}\,d \sigma \\
&~~ = \int_s^t \left[ |\nabla \tunsi|_{L^2}^2- |\nabla \tusi |_{L^2}^2  \right] \langle \tusi, \psi\rangle_{\mathrm{H}}\,d \sigma + \int_s^t |\nabla \tunsi |_{L^2}^2 \langle \tunsi - \tusi, \psi \rangle_{\mathrm{H}}\,d \sigma\\
&~~~~ + \int_s^t |\nabla \tunsi|_{L^2}^2  \langle \tunsi, P_n \psi - \psi\rangle_{\mathrm{H}}~d \sigma \\
&~~ =  \int_s^t \left[ |\nabla \tunsi|_{L^2} - |\nabla \tusi |_{L^2} \right] \left[|\nabla \tunsi |_{L^2} + |\nabla \tusi |_{L^2}\right] \langle \tusi, \psi\rangle_{\mathrm{H}}\,d \sigma\\
&~~~~ +  \int_s^t |\nabla \tunsi |_{L^2}^2 \langle \tunsi - \tusi, \psi\rangle_{\mathrm{H}}\,d \sigma + \int_s^t  |\nabla \tunsi|_{L^2}^2  \langle\tunsi, P_n \psi - \psi\rangle_{\mathrm{H}}~d \sigma.
\end{align*}
Thus by the Cauchy-Schwarz inequality we get
\begin{align*}
&\int_s^t |\nabla \tunsi|_{L^2}^2 \langle \tunsi, P_n \psi\rangle_{\mathrm{H}}\,d \sigma - \int_s^t |\nabla \tusi |_{L^2}^2\langle \tusi, \psi\rangle_{\mathrm{H}}\,d \sigma \\
&~~ \leq  \int_s^t \left[ \|\tunsi - \tusi\|_{\mathrm{V}} \right] \left[\|\tunsi\|_{\mathrm{V}} + \|\tusi\|_{\mathrm{V}}\right] |\tusi|_{\mathrm{H}} |\psi|_{\mathrm{H}}\,d \sigma\\
&~~~~ +  \int_s^t \|\tunsi\|_{\mathrm{V}}^2 |\tunsi - \tusi|_{\mathrm{H}} |\psi|_{\mathrm{H}}\,d \sigma + \int_s^t \|\tunsi\|_{\mathrm{V}}^2 |\tunsi|_{\mathrm{H}} |P_n \psi - \psi|_{\mathrm{H}}\,d \sigma
\end{align*}
By \eqref{eq:5.36}, since $\tilde{u}_n \rightarrow \tilde{u}$ strongly in $\ccal([0,T]; \rH) \cap L^2(0,T; \rV)$, in particular $\tilde{u} \in L^2(0,T; \rV)$, also the sequence ($\tilde{u}_n$) is uniformly bounded in $L^2(0,T; \rV)$ and $P_n \psi \to \psi$ in $\rH$. Thus we have $\tilde{\mathbb{P}}-$a.s.
\begin{align*}
&\lim_{n \to \infty}\int_s^t |\nabla \tunsi|_{L^2}^2 \langle \tunsi, P_n \psi\rangle_{\mathrm{H}}\,d \sigma - \int_s^t |\nabla \tusi |_{L^2}^2\langle \tusi, \psi\rangle_{\mathrm{H}}\,d \sigma \\
&~~ \leq \lim_{n \to \infty} \left[ |\tilde{u}_n|_{L^2(0,T; V)} + |\tilde{u}|_{L^2(0,T; V)} \right]|\tilde{u}_n|_{L^{\infty}(0,T; H)}|\tilde{u}_n - \tilde{u}|_{L^2(0,T; V)}|\psi|_{\mathrm{H}}\\
&~~~~ + \lim_{n \to \infty} |\tilde{u}_n|_{L^2(0,T; V)}^2 |\tilde{u}_n - \tilde{u}|_{L^{\infty}(0,T; H)} | \psi|_\rH + \lim_{n \to \infty} |\tilde{u}_n|_{L^2(0,T; V)}^2 |\tilde{u}_n|_{L^{\infty}(0,T; H)}|P_n \psi - \psi|_\rH\rightarrow 0.
\end{align*}
Hence we infer that assertion $(d)$ holds.\\

Now we are left to show that $(e)$ holds. Let $\psi \in \rH$, then
\begin{align*}
&\int_s^t \langle C^2 \tunsi, P_n \psi\rangle_{\mathrm{H}}\,d \sigma - \int_s^t\langle C^2 \tusi, \psi \rangle_{\mathrm{H}}\,d \sigma\\
&~~ = \int_s^t \langle C^2(\tunsi - \tusi), \psi\rangle_{\mathrm{H}}\,d \sigma + \int_s^t \langle C^2 \tunsi, P_n \psi - \psi\rangle_{\mathrm{H}}~d \sigma \\
&~~ \leq \int_s^t \langle C^2A^{-1}\mathrm{A} (\tunsi - \tusi), \psi\rangle_{\mathrm{H}} \,d \sigma + K_c^2 \int_s^t |\tunsi|_{\mathrm{D}(\mathrm{A})} |P_n \psi - \psi|_{\mathrm{H}}\,d \sigma,
\end{align*}
where $K_c$ is defined in \eqref{eqn-K_c}.\\
Since $(\tilde{u}_n)$ is a uniformly bounded sequence in $L^2(0,T; \mathrm{D}(\mathrm{A}))$ and $C^2 \mathrm{A}^{-1}$ is a bounded operator thus by \eqref{eq:5.36}, we have $\tp$-a.s.
\begin{align*}
&\lim_{n \to \infty} \int_s^t \langle C^2 \tunsi, P_n \psi\rangle_{\mathrm{H}}\,d \sigma - \int_s^t\langle C^2 \tusi, \psi)_{\mathrm{H}}\,d \sigma\\
&~~ \leq \lim_{n \to \infty} \int_s^t \langle \mathrm{A}(\tunsi - \tusi), (C^2 \mathrm{A}^{-1})^{\ast} \psi\rangle_{\mathrm{H}}\,d \sigma + \lim_{n \to \infty} K_c^2 |\tilde{u}|_{L^2(0,T; \mathrm{D}(\mathrm{A}))}|P_n \psi - \psi|_\rH T^{1/2}\\
& ~~ = \lim_{n \to \infty} \int_s^t \langle \tunsi - \tusi, \mathrm{A}^{-1}(C^2 \mathrm{A}^{-1})^{\ast} \psi\rangle_{\mathrm{D}(\mathrm{A})}\,d \sigma + \lim_{n \to \infty} K_c^2 |\tilde{u}|_{L^2(0,T; D(\mathrm{A}))}|P_n \psi - \psi|_\rH T^\frac12 \to 0,
\end{align*}
where to establish the convergence we have used that $P_n \psi \to \psi$ in $\rH$. This completes the proof of Lemma~\ref{lemma5.8}.
\end{proof}

Let $h$ be a bounded continuous function on $\ccal([0,T]; \rH)$ and $\tilde{\mathbb{F}} =\left(\tilde{\mathcal{F}}_t\right) = \sigma \left\{\tilde{u}(s), s \le t\right\}$ be the filtration of sigma fields generated by the process $\tilde{u}$.\\

\begin{lemma}
\label{lemma5.9}
For all $s, t \in [0,T]$, such that $s \leq t$ and all $\psi \in \mathrm{V} \colon$
\begin{equation}
\label{eq:conv_mart}
 \lim_{n \rightarrow \infty} \tilde{\Eb} \left[ \langle\tilde{M}_n(t) - \tilde{M}_n(s), \psi\rangle h(\tilde{u}_{n|[0,s]})\right] = \tilde{\Eb} \left[ \langle\tilde{M}(t) - \tilde{M}(s), \psi\rangle h(\tilde{u}_{|[0,s]}) \right].
 \end{equation}
Here $\langle \cdot, \cdot \rangle$ denotes the duality between $\rV$ and $\rV^\prime$.
\end{lemma}

\begin{proof}
Let us fix $s, t \in [0,T], s \leq t$ and $\psi \in \mathrm{V}$. By \eqref{eq:5.31}, we have
\begin{align*}
&\langle\tilde{M}_n(t) - \tilde{M}_n(s), \psi\rangle = \langle \tilde{u}_n(t), P_n \psi\rangle_H - \langle\tilde{u}_n(s), P_n \psi \rangle_{\mathrm{H}} + \int_s^t \langle \mathrm{A} \tunsi, P_n \psi \rangle_{\mathrm{H}}~d \sigma\\
&~~~~  + \int_s^t \langle B(\tunsi), P_n \psi \rangle \,d \sigma - \int_s^t |\nabla \tunsi|_{L^2}^2 \langle \tunsi, P_n \psi \rangle_{\mathrm{H}}~d \sigma \\
&~~~~ - \dfrac12  \int_s^t \langle C^2 \tunsi, P_n \psi \rangle_{\mathrm{H}}~d \sigma.
\end{align*}
By Lemma \ref{lemma5.8}, we infer that
\begin{equation}
\label{eq:5.37}
\lim_{n \rightarrow \infty} \langle \tilde{M}_n(t) - \tilde{M}_n(s), \psi \rangle = \langle \tilde{M}(t) - \tilde{M}(s), \psi \rangle, \quad \tilde{\mathbb{P}}\text{-a.s.}
\end{equation}
In order to prove \eqref{eq:conv_mart} we first observe that since $\tun \to \tu$ in $\mathcal{Z}_T$, in particular in $\ccal([0,T]; \rH)$ and $h$ is a bounded continuous function  on $\ccal([0,T]; \rH)$, we get
\begin{equation}
\label{eq:h_conv}
\lim_{n \rightarrow \infty} h(\tilde{u}_{n|[0,s]}) = h(\tilde{u}_{|[0,s]})~~~\tilde{\mathbb{P}}-a.s.
\end{equation}
and
\begin{equation}
\label{eq:h_bounded}
\sup_{n \in \mathbb{N}} |h(\tilde{u}_{n|[0,s]})|_{L^{\infty}} < \infty.
\end{equation}
Let us define a sequence of $\R-$valued random variables:
\[ f_n(\omega):= \left[ \langle \tilde{M}_n(t, \omega), \psi \rangle - \langle \tilde{M}_n(s, \omega), \psi \rangle \right]h(\tilde{u}_{n|[0,s]}), ~~~~ \omega \in \tilde{\Omega}.\]
\noindent We will prove that the functions $\{f_n\}_{n \in \mathbb{N}}$ are uniformly integrable in order to apply the Vitali theorem later on. We claim that
\begin{equation}
\label{eq:5.38}
\sup_{n \geq 1} \tilde{\Eb}[|f_n|^2] < \infty.
\end{equation}
By the Cauchy-Schwarz inequality and the embedding $\rV^\prime \hookrightarrow \rH$, for each $n \in \mathbb{N}$ there exists a positive constant $c$ such that
\begin{equation}
\label{eq:5.39}
\tilde{\Eb} [|f_n|^2] \leq 2c |h|^2_{L^{\infty}}|\psi|^2_{\mathrm{V}} \tilde{\Eb}\left[|\tilde{M}_n(t)|^2_\rH + |\tilde{M}_n(s)|^2_\rH \right].
\end{equation}
Since $\tilde{M}_n$ is a continuous martingale with quadratic variation defined in \eqref{eq:5.32}, by the Burkholder-Davis-Gundy inequality we obtain
\begin{equation}
\label{eq:5.40}
\tilde{\Eb} \left[ \sup_{t \in [0,T]} |\tilde{M}_n(t)|_{\mathrm{H}}^2 \right] \leq c \tilde{\Eb} \left[  \sum_{j=1}^m \int_0^T |P_n C_j \tunsi|_\rH^2 \,d \sigma \right].
\end{equation}
Since $P_n \colon \rH \to \rH$ is a contraction then by Assumption $(A.1)$ and \eqref{eq:5.27} for $p = 1$, we have
\begin{align}
\label{eq:5.41}
\tilde{\Eb} \left[\sum_{j=1}^m \int_0^T |P_n C_j \tunsi|_\rH^2~d \sigma  \right] & \leq \tilde{\Eb} \left[ m K_c^2 \int_0^T \|\tunsi\|_{\mathrm{V}}^2\,d \sigma\right]\nonumber \\
& \leq m K_c^2 \tilde{\Eb}\left[ \sup_{\sigma \in [0,T]}\|\tunsi\|_{\mathrm{V}}^2\right]T < \infty.
\end{align}
Then by \eqref{eq:5.39} and \eqref{eq:5.41} we see that \eqref{eq:5.38} holds. Since the sequence $\{f_n\}_{n \in \mathbb{N}}$ is uniformly integrable and by \eqref{eq:5.37} it is $\tilde{\mathbb{P}}-$a.s. point-wise convergent, then application of the Vitali theorem completes the proof of the lemma.
\end{proof}

From Lemma~\ref{lemma5.7} and Lemma~\ref{lemma5.9} we have the following corollary.
\begin{corollary}
\label{cor_martingale}
For all $s,t \in [0,T]$ such that $s \leq t \colon$
\[\E \left(\tilde{M}(t) - \tilde{M}(s) \big|\tilde{\mathcal{F}}_t \right) = 0\,.\]
\end{corollary}

\begin{lemma}
\label{lemma5.10}
For all $s,t \in [0,T]$ such that $s \leq t$ and all $\psi, \zeta \in \rV$:
\begin{align*}
\lim_{n \rightarrow \infty} & \tilde{\Eb} \Big[ \Big( \langle \tilde{M}_n(t), \psi \rangle \langle \tilde{M}_n(t), \zeta \rangle - \langle \tilde{M}_n(s), \psi \rangle \langle \tilde{M}_n(s), \zeta \rangle \Big) h(\tilde{u}_{n|[0,s]}) \Big] \\
& = \tilde{\Eb} \Big[\Big( \langle \tilde{M}(t), \psi \rangle \langle \tilde{M}(t), \zeta \rangle - \langle \tilde{M}(s), \psi \rangle \langle \tilde{M}(s), \zeta \rangle \Big) h(\tilde{u}_{|[0,s]}) \Big],
\end{align*}
where $\langle \cdot, \cdot \rangle$ denotes the dual pairing between $\rV^\prime$ and $\rV$.
\end{lemma}

\begin{proof}
Let us fix $s, t \in [0,T]$ such that $s \leq t$ and $\psi, \zeta \in \rV$ and define the random variables $f_n$ and $f$ by
\begin{align*}
&f_n(\omega) := \Big( \langle \tilde{M}_n(t, \omega), \psi \rangle \langle \tilde{M}_n(t, \omega), \zeta \rangle - \langle \tilde{M}_n(s, \omega), \psi \rangle \langle \tilde{M}_n(s, \omega), \zeta \rangle \Big) h(\tilde{u}_{n|[0,s]}(\omega)),\\
&f(\omega) := \Big( \langle \tilde{M}(t, \omega), \psi \rangle \langle \tilde{M}(t, \omega), \zeta \rangle - \langle \tilde{M}(s, \omega), \psi \rangle \langle \tilde{M}(s, \omega), \zeta \rangle \Big) h(\tilde{u}_{|[0,s]}(\omega)),~~~~ \omega \in \tilde{\Omega}.
\end{align*}
By \eqref{eq:5.37} and \eqref{eq:h_conv} we infer that $\lim_{n \rightarrow \infty} f_n(\omega) = f(\omega)$, for $\tilde{\mathbb{P}}$ almost all $\omega \in \tom$.\\
We will prove that the functions $\{f_n\}_{n \in \mathbb{N}}$ are uniformly integrable. We claim that for some $r > 1$,
\begin{equation}
\label{eq:5.42}
\sup_{n \ge 1} \tilde{\Eb} \left[|f_n|^r\right] < \infty.
\end{equation}
For each $n \in \mathbb{N}$, as before we have
\begin{equation}
\label{eq:5.43}
\tilde{\Eb} \left[ |f_n|^r\right] \leq C \|h\|_{L^{\infty}}^r \|\psi\|^r_{\mathrm{V}} \|\zeta\|^r_{\mathrm{V}} \tilde{\Eb} \left[|\tilde{M}_n(t)|^{2r} + |\tilde{M}_n(s)|^{2r} \right].
\end{equation}
Since $\tilde{M}_n$ is a continuous martingale with quadratic variation defined in \eqref{eq:5.31}, by the Burkholder-Davis-Gundy inequality we obtain
\begin{equation}
\label{eq:5.44}
\tilde{\Eb} \left[ \sup_{t \in [0,T]} |\tilde{M}_n(t)|^{2r} \right] \leq c \tilde{\Eb} \left[ \sum_{j=1}^m \int_0^T |P_n C_j \tunsi|_{\mathrm{H}}^2\,d \sigma \right]^r.
\end{equation}
Since $P_n \colon \rH \to \rH$ is a contraction, by Assumption $(A.1)$ we have
\begin{align}
\label{eq:5.45}
\tilde{\Eb} \left[ \sum_{j=1}^m \int_0^T |P_n C_j \tunsi|_{\mathrm{H}}^2\,d\sigma \right]^r & \leq \tilde{\Eb} \left[ m K_c^2 \int_0^T \|\tunsi\|_{\mathrm{V}}^2\,d \sigma \right]^r\nonumber \\
& \leq (m K_c^2)^r \, \tilde{\Eb}\left(\sup_{\sigma \in [0,T]}\|\tunsi\|_{\mathrm{V}}^{2r}\right) T^r.
\end{align}
Thus for $r \in (1, 1 + \frac{1}{K_c^2})$, by \eqref{eq:5.43}, \eqref{eq:5.44}, \eqref{eq:5.45} and \eqref{eq:5.27} we infer that condition \eqref{eq:5.42} holds. Hence, by the Vitali theorem we infer that 
\begin{equation}
\label{eq:5.46}
\lim_{n \rightarrow \infty} \tilde{\Eb}[f_n] = \tilde{\Eb}[f].
\end{equation}
The proof of the lemma is thus complete.
\end{proof}

\begin{lemma}[Convergence of the quadratic variations]
\label{lemma5.11} For any $s, t \in [0, T]$ and $\psi, \zeta \in \rV$, for all $h \in \ccal([0,T];\rH)$ we have
\begin{align*}
\lim_{n \rightarrow \infty} & \tilde{\Eb}\left[ \left( \sum_{j=1}^m \int_s^t \left\langle \left(C_j \tunsi \right)^\ast P_n \psi , \left(C_j \tunsi\right)^\ast P_n \zeta \right\rangle_{\R} d\sigma \right) \cdot h(\tilde{u}_{n|[0,s]})\right]\\
& = \tilde{\Eb} \left[ \left( \sum_{j=1}^m \int_s^t \left\langle \left(C_j \tusi\right)^\ast \psi , \left(C_j \tusi\right)^\ast \zeta \right\rangle_{\R} \,d \sigma \right) \cdot h(\tilde{u}_{|[0,s]})\right].
\end{align*}
\end{lemma}

\begin{proof}
Let us fix $\psi, \zeta \in \rV$ and define a sequence of random variables by
\[ f_n(\omega) := \left( \sum_{j=1}^m \int_s^t \left \langle \left(C_j \tilde{u}_n(\sigma, \omega)\right)^\ast P_n \psi , \left(C_j \tilde{u}_n(\sigma, \omega)\right)^\ast P_n \zeta \right\rangle_{\R} d \sigma \right) \cdot h(\tilde{u}_{n|[0,s]}),~~~ \omega \in \tilde{\Omega}.\]
We will prove that these random variables are uniformly integrable and convergent $\tilde{\mathbb{P}}-$a.s. to some random variable $f$. In order to do that we will show that for some $r > 1$,
\begin{equation}
\label{eq:quad1}
\sup_{n \geq 1} \tilde{\Eb}\,|f_n|^r < \infty.
\end{equation}

Since $P_n \colon \rH \to \rH$ is a contraction, by the Cauchy-Schwarz inequality, and Assumption $(A.1)$ there exists a positive constant $c$ such that
\begin{align*}
 \left|\left(C_j\tilde{u}_n(\sigma, \omega) \right)^\ast P_n \psi\right|_\R & \leq \left|\left(C_j \tilde{u}_n(\sigma, \omega)\right)^\ast\right|_{L(\rH;\R)} |P_n \psi|_\rH \leq |C_j \tilde{u}_n(\sigma, \omega)|_{L(\R;\rH)} |\psi|_\rH \\
 & \leq K_c \,\|\tilde{u}_n(\sigma, \omega)\|_{\mathrm{V}} |\psi|_{\mathrm{H}}, \quad \quad \quad \quad j \in \{1, \cdots, m\},
 \end{align*}
where $L(X,Y)$ denotes the operator norm of the linear operators from $X$ to $Y$. Thus using the H\"older inequality, we obtain
\begin{align}
\label{eq:quad2}
\tilde{\E}\,|f_n|^r &= \tilde{\E}\,\left| \left( \sum_{j=1}^m \int_s^t \left\langle \left(C_j \tunsi\right)^\ast P_n \psi ,\left(C \tunsi\right)^\ast P_n \zeta \right\rangle_\R d \sigma \right) \cdot h(\tilde{u}_{n|[0,s]})\right|^r \nonumber \\
& \leq |h|^{r}_{{L}^\infty} \tilde{\E}\, \left( \sum_{j=1}^m \int_s^t \left| \left(C_j\tilde{u}_n(\sigma) \right)^\ast P_n \psi\right|_{\R} \cdot \left| \left(C_j \tilde{u}_n(\sigma) \right)^\ast P_n \zeta\right|_{\R} \, d\sigma \right)^r \nonumber \\
& \leq (m K_c^2)^r\, |h|_{{L}^\infty} ^{r}|\psi|_{\mathrm{H}}^r |\zeta|_{\mathrm{H}}^r \tilde{\E} \left(\int_s^t \|\tilde{u}_n(\sigma)\|_{\mathrm{V}}^2 \, d \sigma \right)^r \nonumber \\
& \le (m K_c^2)^r\, |h|_{{L}^\infty} ^{r}|\psi|_{\mathrm{H}}^r |\zeta|_{\mathrm{H}}^r \tilde{\E}\left(\sup_{\sigma \in [0,T]} \|\tunsi\|^{2r}_\rV \right)T^r.
\end{align}
Therefore using \eqref{eq:quad2} and \eqref{eq:5.27} we infer that \eqref{eq:quad1} holds for every $r \in (1, 1+ \frac{1}{K_c^2})$.

Now for pointwise convergence we will show that for a fix $\omega \in \tilde{\Omega}$,
\begin{align}
\label{eq:quad3}
\lim_{n \rightarrow \infty} & \int_s^t \sum_{j=1}^m \left\langle \left(C_j \tilde{u}_n(\sigma, \omega)\right)^\ast P_n \psi, \left(C_j \tilde{u}_n(\sigma, \omega)\right)^\ast P_n \zeta \right\rangle_{\R} \, d\sigma \\ &= \int_s^t \sum_{j=1}^m \left\langle \left(C_j \tilde{u}(\sigma, \omega) \right)^\ast \psi, \left(C_j \tilde{u}(\sigma, \omega)\right)^\ast \zeta \right\rangle_{\R} \, d \sigma.
\nonumber
\end{align}
Let us fix $\omega \in \tom$ such that
\begin{itemize}
\item[(i)] $\tilde{u}_n(\cdot, \omega) \to \tilde{u}(\cdot, \omega)$ in $L^2(0,T; \rV)$,

\item[(ii)] and the sequence $(\tilde{u}_n(\cdot, \omega))_{n \ge 1}$ is uniformly bounded in $L^2(0,T; \rV)$.
\end{itemize}

Note that to prove \eqref{eq:quad3}, it is sufficient to prove that
\begin{equation}
\label{eq:5.47}
\left(C_j \tilde{u}_n(\sigma, \omega)\right)^\ast P_n \psi \rightarrow \left(C_j \tilde{u}(\sigma, \omega)\right)^\ast \psi~~ \text{in}~ L^2(s,t; \R),
\end{equation}
for every $j \in \{1, \cdots, m\}$. Using Cauchy-Schwarz inequaltiy we have
\begin{align*}
\int_s^t & \left|\left(C_j \tilde{u}_n(\sigma, \omega)\right)^\ast P_n \psi - \left(C_j \tilde{u}(\sigma, \omega)\right)^\ast \psi\right|^2_{\R} \,d\sigma\\
& \leq \int_s^t \Big(\left| \left(C_j \tilde{u}_n(\sigma, \omega)\right)^\ast\left(P_n \psi - \psi\right)\right|_{\R} + \left|\left(C_j \tilde{u}_n(\sigma, \omega) - C_j \tilde{u}(\sigma, \omega) \right)^\ast \psi\right|_{\R} \Big)^2 d\sigma \\
& \leq 2 \int_s^t \left|C_j \tilde{u}_n(\sigma, \omega)\right|^2_{L(\R; \rH)} \left|P_n \psi - \psi\right|^2_{\mathrm{H}} d\sigma + 2 \int_s^t \left|C_j \tilde{u}_n(\sigma, \omega) - C_j \tilde{u}_(\sigma, \omega)\right|_{L(\R;\rH)}^2 |\psi|^2_{\mathrm{H}}\, d \sigma\\
& =: I_n^1(t) + I_n^2(t).
\end{align*}
We will deal with each of the terms individually. We start with $I_n^1(t)$. Since
\[\lim_{n \rightarrow \infty}|P_n \psi - \psi|_{\mathrm{H}} = 0, \quad \psi \in \rV,\]
and by Assumption $(A.1)$, (ii) there exists a positive constant $K$ such that
\[\sup_{ n \ge 1} \int_s^t \left|C \tilde{u}_n(\sigma, \omega)\right|_{L(\R ;\mathrm{H})}^2\,d\sigma \leq K_c^2 \sup_{n \ge 1} \int_s^t \|\tilde{u}_n(\sigma, \omega)\|_{\mathrm{V}}^2\,d\sigma \leq K. \]
Thus we infer
\[\lim_{n \rightarrow \infty}I_n^1(t) = 0.\]

Next we consider $I_n^2(t)$. Using Assumption $(A.1)$ and (i) we can show that for every $j \in \{1, \cdots, m\}$,
\begin{align*}
\lim_{n \to \infty}&\int_s^t \left|C_j \tilde{u}_n(\sigma, \omega) - C_j \tilde{u}_(\sigma, \omega)\right|_{L(\R;\mathrm{H})}^2 |\psi|^2_{\mathrm{H}}\, d \sigma \\
&\leq \lim_{n \to \infty}\, |\psi|^2_{\mathrm{H}}\, K_1\,\int_s^t\|\tilde{u}_n(\sigma, \omega) - \tilde{u}(\sigma, \omega)\|_{\mathrm{V}}^2\, d\sigma = 0.
\end{align*}
Hence, we have proved \eqref{eq:5.47}, finishing the proof of lemma.
\end{proof}

By Lemma \ref{lemma5.9} we can pass to the limit in \eqref{eq:5.33}. By Lemmas \ref{lemma5.10} and \ref{lemma5.11} we can pass to the limit in \eqref{eq:5.34} as well. After passing to the limits we infer that for all $\psi, \zeta \in \rV$:
\begin{equation}
\label{eq:5.48}
\tilde{\Eb} \left[ \langle \tilde{M}(t) - \tilde{M}(s) , \psi \rangle h(\tilde{u}_{|[0,s]}) \right] = 0,
\end{equation}
and
\begin{align}
\label{eq:5.49}
\tilde{\Eb} \Big[ &\Big( \langle \tilde{M}(t), \psi \rangle \langle \tilde{M}(t), \zeta \rangle - \langle \tilde{M}(s), \psi \rangle \langle \tilde{M}(s), \zeta \rangle  \nonumber \\
& - \sum_{j=1}^m\int_s^t \left\langle \left(C_j \tusi\right)^\ast \psi , \left(C_j \tusi\right)^\ast \zeta \right\rangle_\R d \sigma \Big) \cdot h(\tilde{u}_{|[0,s]}) \Big] = 0.
\end{align}

From the two previous lemmas and Lemma~\ref{lemma5.7}, we infer the following corollary.
\begin{corollary}
\label{cor_quadratic}
For $t \in [0,T]$
\[\langle \langle \tilde{M}\rangle \rangle_t = \int_0^t \sum_{j=1}^m \left|C_j \tus\right|^2_\rH \,ds\,,\quad \quad t \in [0,T]\,.\]
\end{corollary}

\noindent \textbf{Theorem~\ref{thm5.1} proof continued.} Now we apply the idea analogous to that used by Da Prato and Zabczyk, see \cite[Section~8.3]{[DZ92]}. By Lemma~\ref{lemma_mart} and Corollary~\ref{cor_martingale}, we infer that $\tilde{M}(t)$, $t \in [0, T]$ is an $\rH$-valued continuous square integrable martingale with respect to the filtration $\tilde{\mathbb{F}} = (\tilde{\mathcal{F}_t})_{t \ge 0}$. Moreover, by Corollary~\ref{cor_quadratic} the quadratic variation of $\tilde{M}$ is given by
\[\langle \langle \tilde{M}\rangle \rangle_t = \int_0^t \sum_{j=1}^m \left|C_j \tus\right|^2_\rH \,ds\,,\quad \quad t \in [0,T]\,.\]
Therefore by the martingale representation theorem, there exist
\begin{itemize}
\item a stochastic basis $(\tilde{\tilde{\Omega}}, \tilde{\tilde{\mathcal{F}}}, \tilde{\tilde{\mathcal{F}}}_{t \geq 0}, \tilde{\tilde{\mathbb{P}}})$,
\item a $\R^m-$valued $\tilde{\tilde{\mathbb{F}}}-$Wiener process $\tilde{\tilde{W}}(t)$ defined on this basis,
\item and a progressively measurable process $\tilde{\tu}(t)$ such that for all $t \in [0,T]$ and $\v \in \rV \colon$
\begin{align*}
&\langle \tilde{\tu}(t), \v \rangle - \langle \tilde{\tilde{u}}(0), \v \rangle + \int_0^t \langle \mathrm{A} \tilde{\tu}(s), \v \rangle\,ds + \int_0^t \langle B(\tilde{\tu}(s)), \v \rangle \, ds\\
&~~ = \int_0^t |\nabla \tilde{\tu}(s)|_{L^2}^2 \langle \tilde{\tu}(s), \v \rangle\,ds + \dfrac{1}{2}\int_0^t \sum_{j=1}^m \langle C_j^2\tilde{\tu}(s), \v \rangle \,ds + \int_0^t \sum_{j=1}^m \langle C_j \tilde{\tu}(s), \v \rangle \,d\tilde{\tilde{W}}(s).
\end{align*}
\end{itemize}
Thus the conditions from Definition~\ref{defn4.1} hold with $(\hat{\Omega}, \hat{\mathcal{F}}, \{\hat{\mathcal{F}}_t\}_{t \geq 0}, \hat{\mathbb{P}}) = (\tilde{\tilde{\Omega}}, \tilde{\tilde{\mathcal{F}}}, \{\tilde{\tilde{\mathcal{F}}}_t\}_{t \geq 0}, \tilde{\tilde{\mathbb{P}}})$, $\hat{W} = \tilde{\tilde{W}}$ and $\hat{u} = \tilde{\tilde{u}}$. Hence the proof of Theorem~\ref{thm5.1} is complete.

\section{Pathwise uniqueness and strong solution}
\label{s:6}

In this section we will show that the solutions of \eqref{eq:4.1} are pathwise unique and that \eqref{eq:4.1} has a strong solution. In the previous section we showed that paths of martingale solution $u$ of \eqref{eq:4.1} belong to $\ccal([0,T]; \rV_{\mathrm{w}}) \cap L^2(0,T; \mathrm{D}(\mathrm{A}))$. We start by proving Lemma~\ref{lemma6.1}, in particular showing $u \in \ccal([0,T]; \rV) \cap L^2(0,T; \mathrm{D}(A))$.

\begin{proof}[Proof of Lemma~\ref{lemma6.1}]
$u$ is a martingale solution of \eqref{eq:4.1} thus, $u \in \ccal ([0,T]; \rV_{\mathrm{w}}) \cap L^2(0,T; \rD(\mathrm{A}))$ $\hp-$a.s. We start by showing that RHS of \eqref{eq:6.2} makes sense. In order to do so we will show that each term on the RHS is well defined.

Firstly we consider the non-linear term arising from Navier-Stokes. Using \eqref{eq:2.3}, the H\"older inequality, \eqref{eq:invariance} and \eqref{eq:6.1}, we have the following bounds $\colon$
\begin{align*}
\hat{\E} & \int_0^T |B(u(s))|^2_\rH\,ds \le 2 \hat{\E} \int_0^T |u(s)|_\rH |\nabla u(s)|^2_{L^2}|u(s)|_{\mathrm{D}(\mathrm{A})}\,ds \\
& \le 2 T^{1/2} \left(\hat{\E} \sup_{s \in [0,T]}\|u(s)\|^4_\rV \right)^{1/2} \left(\hat{\E} \int_0^T |u(s)|^2_{\mathrm{D}(\mathrm{A})}\,ds \right)^{1/2} < \infty.
\end{align*}

Using \eqref{eq:invariance}, the H\"older inequality, \eqref{eq:5.26}, \eqref{eq:5.27} and \eqref{eq:6.1} we have the following inequalities for the non-linear term generated from the projection of the Stokes operator,
\begin{align*}
\hat{\E} \int_0^T \left||\nabla u(s)|^2_{L^2} u(s)\right|^2_\rH\,ds = \hat{\E} \int_0^T|\nabla u(s)|^4_{L^2}\,ds \le T  \left(\hat{\E} \sup_{s \in [0,T]}\|u(s)\|^4_\rV \right) < \infty.
\end{align*}

Next we deal with the correction term arising from the Stratonovich integral. Using Assumption $(A.1)$ and estimate \eqref{eq:6.1}, for every $j \in \{1, \cdots, m\}$ we have
\begin{align*}
\hat{\E} \int_0^T |C_j^2 u(s)|^2_\rH \le K_c^2 \hat{\E} \int_0^T |u(s)|^2_{\mathrm{D}(A)}\,ds < \infty,
\end{align*}
where $K_c$ is defined in equality \eqref{eqn-K_c}.

We are left to show that the It\^o integral belongs to $L^2(\Omega \times [0,T];\rV)$. Due to It\^o isometry it is enough to show that for every $j \in \{1, \cdots, m\}$
\begin{equation}
\label{eq:6.3}
\hat{\E} \int_0^T \|C_j u(s)\|^2_\rV \,ds < \infty.
\end{equation}
Using Assumption $(A.1)$ and estimate \eqref{eq:6.1}, we have
\begin{align*}
\hat{\E} \int_0^T\|C_j u(s)\|_\rV^2\,ds \le K_c \,\hat{\E} \int_0^T |u(s)|^2_{\mathrm{D}(\mathrm{A})}\,ds < \infty.
\end{align*}
Thus we have shown that each term in \eqref{eq:6.2} is well defined. Now we will show that the equality holds.

Since $u$ is a martingale solution of \eqref{eq:4.1}, for every $\v \in \rV$ and $t \in [0,T]$ it satisfies the equality \eqref{eq:4.2}, i.e. $\hp-$a.s.
\begin{equation*}
\begin{split}
&\langle u(t), \v \rangle - \langle u_0, \v \rangle + \int_0^t \langle Au(s), \v \rangle\,ds + \int_0^t \langle B(u(s)), \v \rangle\,ds\\
&=  \int_0^t |\nabla u(s)|_{L^2}^2 \langle u(s), \v \rangle\,ds + \frac12\int_0^t \sum_{j=1}^m \langle C_j^2 u(s), \v \rangle\,ds + \int_0^t \sum_{j=1}^m  \langle C_ju(s),\v\rangle\,d\hat{W}_j(s).
\end{split}
\end{equation*}
Note that the above equation holds true for every $\v \in \mathcal{V}$(as defined in \eqref{eq:2.1}) and hence \eqref{eq:6.2} holds in the distribution sense. But since $\mathcal{V}$ is dense in $\rV$, equality \eqref{eq:6.2} holds true almost everywhere, which justifies Remark~\ref{rem3.5}.

We use \cite[Lemma~4.1]{[Pardoux75]} to prove the first part of the lemma. We work with the $\mathrm{D}(\mathrm{A}) \subset \rV \subset \rH$ space triple. Let us rewrite equation \eqref{eq:6.2} in the following form
\begin{equation*}
u(t) = u_0 + \int_0^t g(s)\,ds + N(t),
\end{equation*}
where $g$ contains all the deterministic terms and $N$ corresponds to the noise term. We have shown that $g \in L^2(\Omega; L^2(0,T; \rH))$ and $N \in L^2(\Omega; L^2(0,T; \rV))$. Thus from \cite[Lemma~4.1]{[Pardoux75]} we infer that $u \in L^2(\Omega; \ccal([0,T]; \rV))$. This concludes the proof of lemma.
\end{proof}

In the following lemma we will prove that the solutions of \eqref{eq:4.1} are pathwise unique. The proof uses the Schmalfuss idea of application of the It\^o formula for appropriate function (see \cite{[Schmalfuss97]}).

\begin{lemma}
\label{lemma6.3}
Assume that the assumptions $(A.1) - (A.2)$ are satisfied. If $u_1, u_2$ are two martingale solutions of \eqref{eq:4.1} defined on the same filtered probability space $(\hat{\Omega}, \hat{\mathcal{F}}, \hat{\mathbb{F}}, \hp)$ then $\hp-$a.s. for all $t \in [0,T]$, $u_1(t) = u_2(t)$.
\end{lemma}

\begin{proof}

Let us denote the difference of the two solutions by $U:= u_1 - u_2$. Then $U$ satisfies the following equation
\begin{align}
\label{eq:6.4}
dU(t) + \left[ \rA U(t) + B(u_2(t)) - B(u_1(t))\right]\,dt & = \left[ |\nabla u_1(t)|_{L^2}^2u_1(t) - |\nabla u_2(t)|_{L^2}^2u_2(t)\right]\,dt \nonumber \\
 & \,\, \, +\, \sum_{j=1}^m C_j U(t) \circ dW_j(t), \quad t \in [0,T].
\end{align}
Let us define the stopping time
\begin{equation}
\label{eq:6.5}
\tau_N \colon = T \wedge \inf{ \{t \in [0,T] : \|u_1(t)\|_\rV^2 \vee \|u_2(t)\|^2_\rV > N \}}, \quad N \in \nat.
\end{equation}
Since $\hat{\mathbb{E}} \left[ \sup_{t \in [0,T]}\|u_i(t)\|^2_{\rV}\right] < \infty, \hp-$a.s. for $i = 1,2$, $\lim_{N \to \infty} \tau_N = T.$

We apply the It\^o formula to the function
\[F(t,x) = e^{-r(t)}|x|^2_\rH, \quad t \in [0,T],\; x \in \rV\]
where $r(t)$, $t \in [0,T]$ is a real valued function which will be defined precisely later in the proof.

Since
\[\dfrac{\partial F}{\partial t } = - r^\prime(t)e^{-r(t)}|x|^2_\rH, \quad  \dfrac{\partial F}{\partial x}(\cdot) = 2e^{-r(t)} \langle x, \cdot \rangle_\rH,\]
we obtain for all $t \in [0,T]$
\begin{align*}
& e^{-r(t \wedge \tau_N)}|U(t \wedge \tau_N)|^2_\rH = \int_0^{t \wedge \tau_N} e^{-r(s)} \left(-r^{\prime}(s)|U(s)|^2_\rH + 2 \langle -\mathrm{A} U(s) + B(u_1(s)) - B(u_2(s)), U(s) \rangle_\rH \right)\,ds \\
&\; +  \int_0^{t \wedge \tau_N} e^{-r(s)} \left(2 \langle |\nabla u_1(s)|_{L^2}^2u_1(s) - |\nabla u_2(s)|_{L^2}^2u_2(s), U(s)\rangle_\rH + \sum_{j=1}^m \langle C_j^2U(s), U(s) \rangle_\rH \right)\,ds \\
&\; + \dfrac12 \int_0^{t \wedge \tau_N} \sum_{j=1}^m Tr \left[ C_jU(s) \dfrac{\partial^2F}{\partial x^2} (C_jU(s))^\ast\right]\,ds + 2 \int_0^{t \wedge \tau_N} e^{-r(s)}  \sum_{j=1}^m \langle C_j U(s), U(s) \rangle_\rH dW(s).
\end{align*}
Thus using the Assumption (A.1), we obtain the following simplified expression
\begin{align*}
e^{-r(t \wedge \tau_N)}& |U(t \wedge \tau_N)|^2_\rH \le \int_0^{t \wedge \tau_N} e^{-r(s)} \left( -r^\prime(s)|U(s)|^2_\rH -2 \|U(s)\|^2_\rV - 2b(U(s), u_1(s), U(s))\right)\,ds \\
&\,\, + 2 \int_0^{t \wedge \tau_N} e^{-r(s)} \left((|\nabla u_1(s)|_{L^2}^2 - |\nabla u_2(s)|_{L^2}^2) \langle u_1(s), U(s) \rangle_\rH + |\nabla u_2(s)|_{L^2}^2 |U(s)|^2_\rH \right)\,ds \\
&\,\, + \int_0^{t \wedge \tau_N} e^{-r(s)}  \sum_{j=1}^m \left(\langle C_j^2U(s), U(s) \rangle_\rH + \dfrac12 \times 2 \langle C_j U(s), C_j U(s) \rangle_\rH \right)\,ds.
\end{align*}
Using \eqref{eq:2.2} and the Cauchy-Schwarz inequality we get
\begin{align*}
e^{-r(t \wedge \tau_N)}&|U(t \wedge \tau_N)|^2_\rH + 2 \int_0^{t \wedge \tau_N} e^{-r(s)} \|U(s)\|^2_\rV\,ds \\
& \le \int_0^{t \wedge \tau_N} e^{-r(s)} \left( -r^\prime(s)|U(s)|^2_\rH + 4 |U(s)|_\rH\|U(s)\|_\rV\|u_1(s)\|_\rV \right)\,ds\\
&\quad + 2 \int_0^{t \wedge \tau_N} e^{-r(s)}  \|U(s)\|_\rV \Big(|\nabla u_1(s)|_{L^2} + |\nabla u_2(s)|_{L^2} \Big)|u_1(s)|_\rH|U(s)|_\rH \,ds\\
&\quad + 2 \int_0^{t \wedge \tau_N} e^{-r(s)}|\nabla u_2(t)|_{L^2}^2|U(s)|^2_\rH\,ds.
\end{align*}
Using the Young inequality we obtain
\begin{align}
\label{eq:6.6}
e^{-r(t \wedge \tau_N)}|U(t \wedge \tau_N)|^2_\rH & + 2 \int_0^{t \wedge \tau_N} e^{-r(s)} \|U(s)\|^2_\rV\,ds \le \int_0^{t \wedge \tau_N} e^{-r(s)} \left[ -r^\prime(s) + 8 \|u_1(s)\|^2_\rV \right]|U(s)|^2_\rH \,ds \nonumber \\
& + 2 \int_0^{t \wedge \tau_N} e^{-r(s)}  \big(|\nabla u_1(s)|_{L^2} + |\nabla u_2(s)|_{L^2}\big)^2|u_1(s)|^2_\rH |U(s)|^2_\rH \,ds \nonumber \\
& \quad + \int_0^{t \wedge \tau_N} e^{-r(s)} \|U(s)\|^2_\rV\,ds.
\end{align}

Now choosing
\[r(t) := \int_0^t \left[ 8 \|u_1(s)\|^2_\rV + 2 \big(|\nabla u_1(s)|_{L^2} + |\nabla u_2(s)|_{L^2}\big)^2|u_1(s)|^2_\rH \right]\,ds,\]
inequality \eqref{eq:6.6} reduces to
\begin{align*}
e^{-r(t \wedge \tau_N)}|U(t \wedge \tau_N)|^2_\rH +  \int_0^{t \wedge \tau_N} e^{-r(s)} \|U(s)\|^2_\rV\,ds  \le 0.
\end{align*}
In particular
\begin{equation}
\label{eq:6.7}
\sup_{t \in [0,T]} \left[e^{-r(t \wedge \tau_N)}|U(t \wedge \tau_N)|^2_\rH \right] = 0.
\end{equation}

Note that since $u_1$ and $u_2$ are the martingale solutions of \eqref{eq:4.1} satisfying the estimates \eqref{eq:5.5} and \eqref{eq:5.7} and because of the Lemma~\ref{lemma5.3}, $r$ is well defined for all $t \in [0,T]$.

Since $\hp-$a.s. $\lim_{N \to \infty}\tau_N = T$ and $\hat{\E} \left[r(T)\right] < \infty$, thus from \eqref{eq:6.7} we infer that $\hp-$a.s. for all $t \in [0,T]$, $U(t) = 0$. The proof of the lemma is thus complete.
\end{proof}

\begin{definition}
\label{defn6.4}
Let $(\Omega^i, \mathcal{F}^i, \mathbb{F}^i, \mathbb{P}^i, W^i, u^i)$, $i = 1,2$ be the martingale solutions of \eqref{eq:4.1} with $u^i(0) = u_0$, $i=1,2$. Then we say that the solutions are \textbf{unique in law} if
\[\mathrm{Law}_{\mathbb{P}^1}(u^1) = \mathrm{Law}_{\mathbb{P}^2}(u^2)\,\mbox{on}\, \ccal([0, \infty); \rV_{\mathrm{w}}) \cap L^2([0, \infty); \rD(\mathrm{A})),\]
where $\mathrm{Law}_{\mathbb{P}^i}(u^i)$, $i = 1,2$ are by definition probability measures on $\ccal([0, \infty); \rV_{\mathrm{w}}) \cap L^2([0, \infty); \rD(\mathrm{A}))$.
\end{definition}

\begin{corollary}
\label{cor6.5}
Assume that assumptions $(A.1)-(A.2)$ are satisfied. Then
\begin{itemize}
\item[\rm{(1)}] There exists a pathwise unique strong solution of \eqref{eq:4.1}.

\item[\rm{(2)}] Moreover, if $(\Omega, \mathcal{F}, \mathbb{F}, \mathbb{P}, W, u)$ is a strong solution of \eqref{eq:4.1} then for $\mathbb{P}-$almost all $\omega \in \Omega$ the trajectory $u(\cdot, \omega)$ is equal almost everywhere to a continuous $\rV-$valued function defined on $[0,T]$.

\item[\rm{(3)}] The martingale solution of \eqref{eq:4.1} is unique in law.
\end{itemize}
\end{corollary}

\begin{proof}
By Theorem~\ref{thm5.1} there exists a martingale solution and in the Lemma~\ref{lemma6.3} we showed it is pathwise unique, thus assertion \rm{(1)} follows from \cite[Theorem~2]{[Ondrejat04]}. Assertion \rm{(2)} is a direct consequence of Lemma~\ref{lemma6.1}. Assertion \rm{(3)} follows from \cite[Theorems~2,11]{[Ondrejat04]}.
\end{proof}

Using Theorem~\ref{thm5.1}, Lemma~\ref{lemma6.3} and Corollary~\ref{cor6.5} one can infer Theorem~\ref{thm6.6}.

\begin{remark}
For any bounded Borel function $\varphi \in \mathcal{B}_b(\rV)$ and $t \ge 0$, we define
\begin{equation}
\label{eq:6.9}
(P_t \varphi)(u_0) = \E \left[ \varphi (u(t, u_0)) \right], \quad u_0 \in \rV.
\end{equation}
Then one can show that this family of semigroups is sequentially Feller \cite[Proposition~6.2]{[BMO16]}. In order to prove the existence of invariant measure following the idea from Maslowski-Seidler \cite{[MS99]} one requires to obtain certain boundedness in probability which we haven't been able to establish so far. Thus proving the existence of invariant measure for Stochastic Constrained Navier-Stokes equations on $\T$ is still open.
\end{remark}

\bibliographystyle{plain}

\end{document}